\documentclass{amsart}
\usepackage{amssymb,amscd,verbatim,latexsym}
\usepackage{enumerate}
\usepackage{multicol}
\usepackage[mathscr]{eucal}
\usepackage{graphicx}
\usepackage{epsfig,layout,graphics}

\usepackage{color}

\newcommand\datver[1]{\def\datverp
{\par\boxed{\boxed{\text{Version: #1; Run: \today}}}}}\datver{0.1}

\newcommand\rp{'}

\newcommand\arp[1]{\'{#1}}
\newcommand\dlp{``}
\newcommand\drp{''}

\newcommand\umlaut{\"}

\newcommand{\mfk}{\mathfrak}
\newcommand{\fa}{\mfk{A}}

\newcommand{\Prim}{\operatorname{Prim}}
\newcommand{\Ind}{\operatorname{Ind}}

\newcommand{\CC}{\mathbb C}
\newcommand{\HH}{\mathbb{H}}

\newcommand{\RR}{\mathbb R}

\newcommand{\ZZ}{\mathbb Z}

\newcommand{\CI}{{\mathcal C}^{\infty}}
\newcommand{\CIc}{{\mathcal C}^{\infty}_{\text{c}}}
\newcommand\pa{{\partial}}
\newcommand\oM{\overline{M}}

\newcommand{\ind}{\operatorname{ind}}

\newcommand{\ord}{\operatorname{ord}}
\newcommand{\Hom}{\operatorname{Hom}}
\newcommand{\Diff}{\operatorname{Diff}}
\newcommand{\Symb}{\operatorname{Symb}}
\newcommand{\coker}{\operatorname{coker}}
\newcommand{\cl}{\operatorname{cl}}

\newcommand{\vol}{\operatorname{vol}}
\renewcommand{\div}{\operatorname{div}}
\newcommand\pullback{\sp{\downarrow\downarrow}}

\newcommand{\HP}{\operatorname{HP}}

\newcommand\m[1]{${#1}$}
\newcommand\dm[1]{\begin{equation*}{#1}\end{equation*}}
\newcommand\Kond[2]{\maK\sp{#1}_{#2}}
\newcommand{\sh}[1]{#1^{\sharp}}

\newcommand{\maB}{\mathcal B}
\newcommand{\maC}{\mathcal C}
\newcommand{\maD}{\mathcal D}

\newcommand{\maF}{\mathcal F}
\newcommand{\maG}{\mathcal G}

\newcommand{\maI}{\mathcal I}
\newcommand{\maK}{\mathcal K}

\newcommand{\maP}{\mathcal P}

\newcommand{\maT}{\mathcal T}
\newcommand{\maV}{\mathcal V}
\newcommand{\maW}{\mathcal W}

\newcommand\Dir{\ \backslash \! \!\! \! D}
\newcommand\ede{ \, := \, }

\newtheorem{theorem}{Theorem}[section]
\newtheorem{proposition}[theorem]{Proposition}
\newtheorem{corollary}[theorem]{Corollary}

\newtheorem{lemma}[theorem]{Lemma}

\theoremstyle{definition}
\newtheorem{definition}[theorem]{Definition}
\theoremstyle{remark}
\newtheorem{remark}[theorem]{Remark}
\newtheorem{remarks}[theorem]{Remarks}
\newtheorem{example}[theorem]{Example}

\author[V. Nistor]{Victor Nistor} \address{Universit\'{e} de Lorraine,
  UFR MIM, Ile du Saulcy, CS 50128, 57045 METZ, France and
    Inst. Math. Romanian Acad.  PO BOX 1-764, 014700 Bucharest
   Romania} 
\email{victor.nistor@univ-lorraine.fr} 

\thanks{The author was partially supported by ANR-14-CE25-0012-01.\\
Manuscripts available from {\bf
  http:{\scriptsize//}www.math.univ-metz.fr{\scriptsize /} $\tilde{}$
  nistor{\scriptsize /}}}

\date\today

\begin{document}

\title[Analysis on singular spaces]{Analysis on singular spaces: Lie manifolds and
  operator algebras}

\begin{abstract} 
We discuss and develop some connections between analysis on singular
spaces and operator algebras, as presented in my sequence of four
lectures at the conference {\em Noncommutative geometry and
  applications}, Frascati, Italy, June 16-21, 2014. Therefore this
paper is mostly a survey paper, but the presentation is new, and there
are included some new results as well. In particular, Sections
\ref{sec3} and \ref{sec4} provide a complete short introduction to
analysis on noncompact manifolds that is geared towards a class of
manifolds--called \dlp Lie manifolds\drp--that often appears in
practice. Our interest in Lie manifolds is due to the fact that they
provide the link between analysis on singular spaces and operator
algebras. The groupoids integrating Lie manifolds play an important
background role in establishing this link because they provide
operator algebras whose structure is often well understood. The
initial motivation for the work surveyed here--work that spans over
close to two decades--was to develop the index theory of stratified
singular spaces. Meanwhile, several other applications have emerged as
well, including applications to Partial Differential Equations and
Numerical Methods. These will be mentioned only briefly, however, due
to the lack of space. Instead, we shall concentrate on the
applications to Index theory.
\end{abstract}

\maketitle

\tableofcontents

\section*{Introduction}

We survey some connections between analysis on singular spaces and
operator algebras, concentrating on applications to Index theory.  The
paper follows rather closely my sequence of four lectures at the
conference {\em Noncommutative geometry and applications}, Frascati,
Italy, June 16-21, 2014. From a technical point of view, the paper
mostly sets up the analysis tools needed for developing a certain
approach to the Index theory of singular and non-compact spaces.
These results were developed by many people over several years. In
addition to these older contributions, we include here also several
new results tying various older results together. In particular, for
the benefit of the reader, we have written the paper in such a way
that the third and fourth sections are to a large extend
self-contained. They include most of the needed proofs, and thus can
be regarded as a very short introduction to analysis on non-compact
manifolds that focuses on applications to Lie manifolds. Lie manifolds
are a class of non-compact manifolds that arise naturally in many
applications involving non-compact and singular spaces.

The main story told by this paper is, briefly, as follows. Some of the
classical Analysis and Index theory results deal with the index of
Fredholm operators. This is rather well understood in the case of
smooth, compact manifolds and in the case of smooth, bounded
domains. By contrast, the non-smooth and non-compact cases are much
less well understood. Moreover, it has become clear that the index
theorems in these frameworks require non-local invariants and (hence)
cyclic homology. The full implementation of this program requires,
however, further algebraic and analytic developments. More
specifically, one important auxiliary question that needs to be
answered is:\ {\em \dlp Which operators on a given non-smooth or
  non-compact space are Fredholm.\drp} A convenient answer to this
question involves Lie manifolds and the Lie groupoids that integrate
them. The techniques that were developed for this purpose have then
proved to be useful also in other mathematical areas, such as Spectral
theory and the Finite Element Method.

Fredholm operators play a central role in this paper for the following
reasons. First of all, the (Fredholm) index is defined only for
Fredholm operators, thus, in order to state an index theorem for the
Fredholm index, one needs to have examples of Fredholm operators. In
fact, the data that is needed to decide whether a given operator is
Fredholm (principal symbol, boundary--or indicial--symbols) is also
the data that is used for the actual computation of the index of that
operator.  Moreover, many interesting quantities (such as the
signature of a compact manifold) are, in fact, the indices of certain
operators. Second, Fredholm operators have been widely used in Partial
Differential Equations (PDEs). For instance, non-linear maps whose
linearization is Fredholm play a central role in the study of
non-linear PDEs. Also, Fredholm operators are useful in determining
the essential spectra of Hamiltonians. Finally, the Fredholm index is
the first obstruction for an operator to be invertible. As an
illustration, this simple last observation is exploited in our
approach to the Neumann problem on polygonal domains (Theorem
\ref{thm.neumann}).  The proof of that theorem relies first on the
calculation of the index of an auxiliary operator, this auxiliary
operator is then shown to be injective, and the final step is to
augment its domain so that it becomes an isomorphism \cite{HMN}.

A certain point on the analysis on singular and non-compact spaces is
worth insisting upon. A typical approach to analysis on singular
spaces--employed also in this paper--is that the analysis on a
singular space happens on the {\em smooth} part of the space, with the
singularities playing the important role of providing the behavior
\dlp at infinity.\drp\ Thus, from this point of view, the analysis on
non-compact spaces is more general than the analysis on singular
spaces. However, for the simplicity of the presentation, we shall
usually discuss only singular spaces, with the understanding that the
results also extend to non-compact manifolds.

Here are the contents of the paper. The first section is devoted to
describing the Index theory motivation for the results presented in
this paper. The approach to Index theory used in this paper is based
on exact sequences of operator algebras.  Thus, in the first section,
we discuss the exact sequences appearing in the Atiyah-Singer index
theorem, in Connes\rp\ index theorem for foliations, and in the
Atiyah-Patodi-Singer (APS) index theorem. The second section is
devoted to explaining the motivation for the results presented in this
paper coming from degenerate (or singular) elliptic partial
differential equations. In that section, we just present some typical
examples of degenerate elliptic operators that suggest how ubiquitous
they are and point out some common structures that have lead to Lie
manifolds, a class of manifolds that is discussed in the third
section. In this third section, we include the definition of Lie
manifolds, a discussion of manifolds with cylindrical ends (the
simplest non-trivial example of a Lie manifold, the one that leads to
the APS framework), a discussion of Lie algebroids and of their
relation to Lie manifolds, and a discussion of the natural metric and
connection on a Lie manifold. The fourth section is a basic
introduction to analysis on Lie manifolds. It begins with discussions
of the needed functions spaces, of \dlp comparison algebras,\drp\ and
of Fredholm conditions. The last section is devoted to applications,
including the formulation of an index problem for Lie manifolds in
periodic cyclic cohomology, an application to essential spectra, an
index theorem for Callias-type operators, and the Hadamard well
posedness for the Poisson problem with Dirichlet boundary conditions
on polyhedral domains.

The four lectures of my presentation at the above mentioned conference
were devoted, each, to one of the following subjects: I. {\em Index
  theory}, II. {\em Lie manifolds}, III. {\em Pseudodifferential
  operators on groupoids,} and IV. {\em Applications}, and are based
mostly on my joint works with Bernd Ammann (Regensburg), Catarina
Carvalho (Lisbon), Alexandru Ionescu (Princeton), Robert Lauter
(Mainz), Anna Mazzucato (Penn State) and Bertrand Monthubert
(Toulouse). Nevertheless, I made an effort to put the results in
context by quoting and explaining other relevant results. I have also
included significant background results and definitions to make the
paper easier to read for non-specialists. I have also tried to
summarize some of the more recent developments. Unfortunately, the
growing size of the paper has finally prevented me from including more
information. Moreover, it was unpractical to provide all the related
references, and I apologize to the authors whose work has not been
mentioned enough.

I would like to thank the Max Planck Institute for Mathematics in
Bonn, where part of this work was completed. Also, I would like to
thank Bernd Ammann, Ingrid and Daniel Belti\c{t}\u{a}, Karsten Bohlen,
Claire Debord, Vladimir Georgescu, Marius M\u{a}ntoiu, Jean Renault,
Elmar Schrohe, and Georges Skandalis for useful comments.

\section{Motivation: Index Theory\label{sec1}}

This paper is devoted in large part to explaining some applications of
Lie manifolds and of their associated operator algebras to analysis on
singular and non-compact spaces. The initial motivation of this author
for studying analysis on singular and non-compact spaces (and hence
also for studying Lie manifolds) comes from Index theory. In this
section, I will describe this initial motivation, while in the next
section, I will provide further motivation coming from degenerate
partial differential equations. Thus, I will not attempt here to
provide a comprehensive introduction to Index Theory, but rather to
motivate the results and constructions introduced in this paper using
it. In particular, I will stress the important role that Fredholm
conditions play for index theorems. In fact, in our approach, both the
index theorem studied and the associated Fredholm conditions rely on
the {\em same} exact sequence discussed in general in the next
subsection. No results in this section are new.

\subsection{An abstract index theorem\label{ssec.abs.ind}}

An approach to Index theory is based on {\em exact sequences} of
algebras of operators. We shall thus consider an abstract exact
sequence
\begin{equation}\label{eq.abs.es}
	0 \to I \to A \to \Symb \to 0 \,,
\end{equation}
in which the algebras involved will be specified in each particular
application. The same exact sequence will be used to establish the
corresponding Fredholm conditions. Typically, \m{A} will be a suitable
{\em algebra of operators} that describes the analysis on a given
(class of) singular space(s). In our presentation, the algebra $A$
will be constructed using Lie algebroids and Lie groupoids. The choice
of the {\em ideal} \m{I} also depends on the particular application at
hand and is not necessarily determined by $A$. In fact, the analysis
on singular spaces distinguishes itself from the analysis on compact,
smooth manifolds in that there will be several reasonable choices for
the ideal $I$.

Often, in problems related to classical analysis--such as the ones
that involve the Fredholm index of operators--the ideal $I$ will be
contained in the ideal of compact operators $\maK$ (on some separable
Hilbert space). In fact, in most applications in this presentation, we
will have $I := A \cap \maK$. We insist, however, that this is not the
only legitimate choice, even if it is the most frequently used one. An
important other example is provided by taking $I$ to be the kernel of
the principal symbol map. As we will see below, in the case of
singular and non-compact spaces, the kernel of the principal symbol
map {\em does not} consist generally of compact operators. This is the
case in the analysis on covering spaces and on foliations, which also
lead naturally to von Neumann algebras \cite{ConnesFol, Jones,
  kordyukovFol, Voiculescu}.

If $I := A \cap \maK$ and \m{P \in A} has an invertible image in
\m{A/I} (that is, it is invertible modulo \m{I}), then the operator
$P$ is Fredholm and a natural and far reaching question to ask then is
to compute $\ind(P) := \dim \ker(P) - \dim \coker(P)$, the {\em
  Fredholm index of} $P$, defined as the difference of the dimensions
of the kernel and cokernel of $P$.  In any case, we see that in order
to formulate an index problem, we need criteria for the relevant
operators to be Fredholm, because it is the condition that $P$ be
Fredholm that guarantees that $\ker(P)$, the kernel of $P$, and
$\coker(P)$, the cokernel of $P$, are finite dimensional. This is also
related to the structure of the exact sequence \eqref{eq.abs.es}.

When the algebra $A$ of the exact sequence \eqref{eq.abs.es} is
defined using groupoids--as is the case in this presentation--then the
structure of the quotient algebra $\Symb := A/I$ is related to the
representation theory of the underlying groupoid. Unfortunately, we
will not have time to treat this important subject in detail, but we
will provide several references in the appropriate places.

The exact sequence \eqref{eq.abs.es} provides us with a boundary (or
index) map
\begin{equation}\label{eq.b.map}
  \pa \, :\, K_1\sp{alg}(\Symb) \, \to\, K_0\sp{alg}(I) \,,
\end{equation}
between {\em algebraic} $K$-theory groups, whose calculation will be
regarded as an index formula for the reasons explained in the
following subsections (see, for instance, Remark
\ref{rem.index}). Thus, in general, {\em the index} of an operator is
an element of a $K_0$ group, which explains why the usual index, which
is an integer, is called {\em the Fredholm index} in this paper.  In
case \m{A} and \m{I} are \m{C\sp{*}}-algebras, the boundary map $\pa$
descends to a map between the corresponding {\em topological}
$K$-theory groups. Moreover, we obtain also a map \m{\pa\rp :
  K_0(\Symb) \to K_1(I)}, acting also between {\em topological}
$K$-theory groups.  The maps \m{\pa} and $\pa\rp$ and the maps
obtained from the functoriality of (topological) \m{K}-groups, give
rises to a six-term exact sequence of \m{K}-groups \cite{Rordam,
  WeggeOlsen}. Unfortunately, often the \m{K}-groups are difficult to
compute, so we need to consider suitable dense subalgebras of
\m{C\sp{*}}-algebras and their cyclic homology (see, for instance,
Subsection \ref{ssec.foliations}).

We begin with a quick introduction to differential and
pseudodifferential operators needed to fix the notation and to
introduce some basic concepts. It is written to be accessible to
graduate students. We then discuss three basic index theorems and
their associated analysis (or exact sequences). These three index
theorems are: the Atiyah-Singer (AS) index theorem, Connes' index
theorem for foliations, and the Atiyah-Patodi-Singer (APS) index
theorem. We will see that, at least from the point of view adopted in
this presentation, Connes' and APS\rp\ frameworks extend the
Atiyah-Singer\rp s framework in complementary directions.

\subsection{Differential and pseudodifferential operators}

We now fix some notation and recall a few basic concepts. On
\m{\RR^n}, we consider the derivations (or, which is the same thing,
  {\em vector fields}) $\pa_j = \frac{\pa}{\pa x_j},$ \m{j=1, \ldots,
    n,} and form the differential monomials $\pa^\alpha :=
  \pa_1^{\alpha_1} \pa_2^{\alpha_2} \ldots \pa_n^{\alpha_n},$ $\alpha
  \in \ZZ_+^n.$ We let \m{|\alpha| := \alpha_1 + \alpha_2 + \ldots +
    \alpha_n \in \ZZ_+}.
A {\em differential operator} of order \m{m} on \m{\RR^n} is then an
operator $P : \CIc(\RR^n) \to \CIc(\RR^n)$ of the form
\begin{equation}\label{eq.def.diff}
	Pu \, = \, \sum_{|\alpha| \le m} a_\alpha \pa^\alpha u \, ,
\end{equation}
with \m{m} minimal with this property.  Sometimes $P : \CIc(\RR^n) \to
L\sp{2}(\RR^n)$, but in this paper we consider only operators $P$
having {\em smooth} coefficients $a_{\alpha}$.

It is easy, but important, to extend the above constructions to
systems of differential operators, in order to account for important
operators such as: vector Laplacians, elasticity, signature, Maxwell,
and many others. We then take $u = (u_1, \ldots , u_k) \in
\CIc(\RR^n)^k = \CIc(\RR^n; \RR^k)$ to be a smooth, compactly
supported section of the trivial vector bundle \m{\underline{\RR^k} =
  \RR^k \times \RR^n \to \RR^n} on \m{\RR^n} and we take $ a_\alpha
\in \CI(\RR^n; M_k(\RR)) $ to be a matrix valued function. Each
coefficient $a_{\alpha}$ is hence {\em an endomorphism} of the trivial
vector bundle \m{\underline{\RR^k}}. Then $P$ maps $\CIc(\RR^n;
\RR\sp{k})$ to $\CIc(\RR^n; \RR\sp{k}).$
Let \m{\Delta = - \pa_1^2 - \ldots - \pa_n^2 \ge 0} and \m{s \in
  \ZZ_+}. We denote as usual
\begin{equation*}
	H^s(\RR^n) \, := \, \{ u : \RR^n \to \CC, \ \pa^\alpha u \in
        L^2(\RR^n), \, |\alpha| \le s \, \} \, = \,
        \maD(\Delta^{s/2})\,.
\end{equation*}
As we will see below, both definitions above of Sobolev spaces extend
to the case of ``Lie manifolds.'' These definitions of Sobolev spaces
also extend immediately to vector valued functions and, if the
coefficients \m{a_\alpha} of \m{P} are bounded (together with enough
derivatives, more precisely, if \m{P \in W^{s, \infty}(\RR\sp{n})}),
then we obtain that \m{P} maps $H^s(\RR^n)$ to $H^{s-m}(\RR^n)$.

For \m{P} a differential operator of order \m{\le m} as in Equation
\eqref{eq.def.diff}, we let
\begin{equation}\label{eq.def.princs}
	\sigma_m(P)(x, \xi) \, = \ \sum_{|\alpha| = m} a_\alpha(x)
        (\imath \xi)^\alpha \, \in \, \CI(\RR^n \times \RR^n; M_k) \,,
\end{equation}
and call it the {\em principal symbol} of \m{P}. In particular, we have
\m{\sigma_{m+1}(P) = 0}. Here \m{x \in\RR^n} and \m{\xi \in \RR^n} is
the {\em dual} variable.

The fact that \m{\xi} is a dual variable to \m{x \in \RR^n} is
confirmed by the formula for transformations of coordinates. The
principal symbol is thus seen to be a function on \m{T^*\RR^n \simeq
  \RR^n \times \RR^n}. It turns out that the principal symbol
\m{\sigma_m(P)} of \m{P} has a much simpler transformation formula
than the {\em (full) symbol} $\overline{\sigma}(P)$ of \m{P} defined
by
\begin{equation}\label{eq.def.fulls}
	\overline{\sigma}(P)(x, \xi) \, = \ \sum_{|\alpha| \le m}
        a_\alpha(x) (\imath \xi)^\alpha \, \in \, \CI(\RR^n \times
        \RR^n; M_k) \,.
\end{equation}
The full symbol \m{p(x, \xi) := \overline{\sigma}(P)(x, \xi)} of \m{P}
defined as above in Equation \eqref{eq.def.fulls} is nevertheless
important because $P = p(x, D)$, where
\begin{equation}\label{eq.def.pseudodiff}
	p(x, D) u(x) \, := \, (2 \pi)^{-n} \int_{\RR^n} e^{\imath x
          \cdot \xi} p(x, \xi) \hat u(\xi) d\xi \,.
\end{equation}

There exist more general classes of functions (or symbols) \m{p} for
which \m{p(x, D)} can still be defined by the above formula
\eqref{eq.def.pseudodiff}. The resulting operator will be a {\em
  pseudodifferential operator} with symbol \m{p}.  Let us recall the
definition of the two most basic classes of symbols for which the
formula \eqref{eq.def.pseudodiff} defining \m{p(x, D)} still makes
sense \cite{hormander3, SchulzeBook91, RodinoBook}.  For simplicity,
we shall consider in the beginning only scalar symbols, although
matrix valued symbols can be handled in a completely similar way. The
first class of symbols $p$ for which the formula
\eqref{eq.def.pseudodiff} still makes sense is the class \m{S_{1,
    0}^m(\RR^k \times \RR^N)}, \m{m \in \RR}. It is defined as the
space of functions \m{a : \RR^{k +N} \to \CC} that satisfy, for any
\m{i, j \in \ZZ_{+}}, the estimate
\begin{equation*}
	| \pa_x ^i \pa_\xi^j a(x, \xi)| \, \le \, C_{i j} (1 +
        |\xi|)^{m-j} \,,
\end{equation*}
for a constant \m{C_{i j} > 0} independent of \m{x} and \m{\xi}.  Of
course, in \eqref{eq.def.pseudodiff}, we take $N = n$.

Let us now introduce {\em classical symbols}. A function \m{a : \RR^k
  \times \RR^N \to \CC} is called {\em eventually homogeneous} of
order \m{s} if there exists \m{M > 0} such that \m{a(x, t\xi) =
  t\sp{s} a(x, \xi)} for \m{|\xi| \ge M} and \m{t \ge 1}. A very
useful class of symbols is $S^m_{\cl}(\RR^{2n})$, defined as the
subspace of symbols $a \in S^m_{1, 0} (\RR^{2n})$ that can be written
as asymptotic series \m{a \sim \sum_{j=0}^\infty a_{m - j}}, meaning
\dm{a - \sum_{j=0}^N a_{m - j} \, \in \, S^{m-N-1}_{1, 0} (\RR^{2n})
  \,,} with \m{a_k \in S^k_{1, 0} (\RR^{2n})} {\em eventually
  homogeneous} of order \m{k}. If \m{a \in S^m_{\cl}(\RR^{2n})}, the
pseudodifferential operator \m{a(x, D)} is called a {\em classical
  pseudodifferential operator} and its principal symbol is defined by
\begin{equation}\label{eq.def.princs2}
	\sigma_m(a(x, D)) \, := \, a_m 
\end{equation}
and is regarded as a smooth, order \m{m} homogeneous function on
\begin{equation*}
  T^*\RR^n \smallsetminus \mbox{ \dlp zero section\drp} \, = \, \RR^{2n}
  \smallsetminus (\RR\sp{n} \times \{0\}) \,.
\end{equation*}
For Index theory, it is generally enough to consider classical
pseudodifferential operators. The reason is that the inverses and
parametrices of classical pseudodifferential operators are again
classical and, if \m{p(x, \xi) := \sum_{|\alpha| \le m} p_\alpha
  (\imath \xi)\sp{\alpha}} is a polynomial in $\xi$ and if we let \m{P
  := p(x, D) = \sum_{|\alpha| \le m} p_\alpha \pa^\alpha}, then \m{P}
is a {\em classical pseudodifferential operator} of order \m{m}.

The definition of a (pseudo)differential operator \m{P} (of order
\m{\le m}) and of its principal symbol \m{\sigma_m(P)} then extend to
manifolds and vector bundles by using local coordinate charts.  To fix
notation, if \m{E \to M} is a smooth vector bundle over a manifold
\m{M}, we shall denote by \m{\Gamma(M; E)} the space of its smooth
sections:
\begin{equation*}
 \Gamma(M; E) \, := \, \{s : M \to E,\, s(x) \in E_x \} \,.
\end{equation*}
Similarly, we shall denote by \m{\Gamma_c(M; E) \subset \Gamma(M; E)}
the subspace of smooth, compactly supported sections of $E$ over
$M$. Sometimes, when no confusion can arise, we denote \m{\Gamma(E) =
  \Gamma(M; E)} and, similarly, \m{\Gamma_c(E) = \Gamma_c(M; E)}.
Getting back to our extension of pseudodifferential operators to
manifolds, we obtain this extension by replacing as follows:
\begin{align*}
    \RR^n & \ \ \leftrightarrow \ \ \ M \, = \mbox{ a smooth manifold
    }\\
    \CIc(\RR^n)^k & \ \ \leftrightarrow \ \ \ \mbox{ sections of a
      vector bundle,}
\end{align*}
which gives for an order $m$ operator \m{P} acting between smooth,
compactly supported sections of \m{E} and \m{F}:
\begin{equation*}
  \begin{gathered}
    P \, : \, \Gamma_c(M; E) \, \to \, \Gamma_c(M; F)\\
    \sigma_m(P) \, \in \, \Gamma(T^*M \smallsetminus \{0\};\ \Hom(E,
    F)) \,.
\end{gathered}
\end{equation*}
Of course, \m{\sigma_m(P)} is homogeneous of order \m{m}. Thus, if
\m{m=0} and if we denote by \m{S\sp{*}M := (T\sp{*}M \smallsetminus
  \{0\})/\RR_{+}\sp{*}} the {\em (unit) cosphere bundle}, then
\m{\sigma_0(P)} identifies with a smooth function on
\m{S\sp{*}M}. (The name \dlp cosphere bundle\drp\ is due to the fact
that, if we choose a metric on \m{M}, then the cosphere bundle
identifies with the set of vectors of length one in \m{T\sp{*}M}.)

The main property of the principal symbol is the {\em multiplicative
  property}
\begin{equation}\label{eq.mult}
  \sigma_{m+m'}(PP') = \sigma_m(P) \sigma_{m'}(P') \, ,
\end{equation}
a property that is enjoyed by its extension to {\em pseudodifferential
  operators} (which are allowed to have negative and non-integer
orders as well).

\begin{definition}\ 
A (classical, pseudo)differential operator \m{P} is called {\em
  elliptic} if its principal symbol is invertible away from the zero
section of \m{T^*M}.
\end{definition}

See \cite{hormander3, plamenevskiiBook, RodinoBook, simanca} for a more complete
discussion of various classes of symbols and of pseudodifferential
operators. See also \cite{aln2, GilkeyBook, Mantoiu2,
  MelroseICM, parenti, SchulzeBook91}.

As a last ingredient before discussing the Fredholm index, we need to
extend the definition of Sobolev spaces to manifolds. To that end, we
consider also a {\em metric} \m{g} on our manifold \m{M} (or a
Lipschitz equivalence class of such metrics) \cite{Aubin,
  HebeyBook}. Then, for a complete manifold \m{M}, the Sobolev spaces
are given by the domains of the powers of the (positive) Laplacian. In
general, this will depend on the choice of the metric~\m{g}.

\subsection{The Fredholm index}

Let now \m{M} be a {\em compact}, smooth manifold, so the Sobolev
spaces \m{H^s(M)} are uniquely defined.  Let also \m{P} be a
(classical, pseudo) differential operator of order \m{\le m} acting
between the smooth sections of the hermitian vector bundles \m{E} and
\m{F}. We denote by \m{H\sp{s}(M; E)} and \m{H\sp{s}(M; F)} the
corresponding Sobolev spaces of sections of these bundles.

Recall that a continuous, linear operator \m{T : X \to Y} acting
between topological vector spaces is {\em Fredholm}, if and only if,
the vector spaces \m{\ker(P) := \{u \in X, Tu = 0 \} } and
\m{\coker(P) := Y/TX} are finite dimensional. One of our {\em model
  results} is then the following classical theorem
\cite{CordesHerman68, Seeley59, Seeley63}.

\begin{theorem}\label{thm.fc.as} 
Let $P$ be an order $m$ pseudodifferential operator acting between the
smooth sections of the bundles \m{E} and \m{F} on the smooth, compact
manifold $M$ and \m{s \in \RR}. Then
\begin{equation*}
  P : H^s(M; E) \to H^{s-m}(M; F) \mbox{ is Fredholm }
  \ \Leftrightarrow \ P \mbox{ is elliptic.}
\end{equation*}
\end{theorem}

Fredholm operators appear all the time in applications (because
elliptic operators are so fundamental). For instance, the theorem
mentioned above is one of the crucial ingredients in the \dlp Hodge
theory\drp\ for smooth compact manifolds, which is quite useful in
Gauge theory.

By the Open Mapping theorem, the invertibility of a continuous, linear
operator \m{P : X_1 \to X_2} acting between two Banach spaces is
equivalent to the condition \m{\dim \ker(P) = \dim \coker(P) = 0}. It
is important then to calculate the {\em Fredholm index} \m{\ind(P)} of
\m{P}, defined by
\begin{equation}
  \ind(P) \, := \, \dim \ker(P) -\dim \coker(P) \,.
\end{equation}
The reason for looking at the Fredholm index rather than looking
simply at the numbers \m{\dim \ker(P)} and \m{\dim \coker(P)} is that
\m{\ind(P)} has better stability properties than these numbers. For
instance, the Fredholm index is homotopy invariant and depends only on
the principal symbol of \m{P}.

\subsection{The Atiyah-Singer index formula}

The index of elliptic operators on smooth, compact manifolds is
computed by the Atiyah-Singer index formula \cite{AS3}:

\begin{theorem}[Atiyah-Singer] 
\label{thm.AS}
Let \m{M} be a compact, smooth manifold and let \m{P} be elliptic,
classical (pseudo)differential operator acting on sections of smooth
vector bundles on \m{M}.  Then
\begin{equation*}
  \ind(P) \, = \, \langle \, ch [\sigma_m(P)] \maT(M)\, ,\, [T^*M] \,
  \rangle \; .
\end{equation*}
\end{theorem}

A suitable orientation has to be, of course, chosen on $T\sp{*}M$.
There are many accounts of this theorem, and we refer the reader for
instance to \cite{GilkeyBook, palais, Seeley63, Troitsky} for more
details. See \cite{ConnesBook, kordyukovFol} for an approach using
non-commutative geometry. Let us nevertheless mention some of the main
ingredients appearing in the statement of this theorem, because they
are being generalized (or need to be generalized) to the non-smooth
case. This generalization is in part achieved by non-commutative
geometry and by analysis on singular spaces. Thus, returning to
Theorem \ref{thm.AS}, the meanings of the undefined terms in Theorem
\ref{thm.AS} are as follows:
\begin{enumerate}[(i)]
\item The principal symbol \m{\sigma_m(P)} of \m{P} defines a
  \m{K}-theory class in \m{K\sp{0}(T\sp{*}M)} (with compact supports)
  by the ellipticity of \m{P} \cite{AS3} and \m{ ch [\sigma_m(P)] \in
    H_c^{even}(T^*M)} is the Chern character of this class.

\item \m{\maT(M) \in H\sp{even}(M) \simeq H\sp{even}(T\sp*M)} is the
  Todd class of \m{M}, so the product \m{ch [\sigma_m(P)] \maT(M)} is
  well-defined in \m{H_c^{even}(T^*M)}.

\item \m{[T^*M] \in H_c^{even}(T^*M)\rp} is the fundamental class of
  \m{T^*M} and is chosen such that no sign appears in the index
  formula.
\end{enumerate}

The AS index formula was much studied and has found a number of
applications. It is based on earlier work of Grothendieck and
Hirzebruch and answers to a question of Gelfand.  One of the main
motivations for the work presented here is the desire to extend the
index formula for compact manifolds (the AS index formula) to the
noncompact and singular cases. To this end, it will be convenient to
use the exact sequence formalism described in Subsection
\ref{ssec.abs.ind}. Namely, the exact sequence \eqref{eq.abs.es}
corresponding to the AS index formula is
\begin{equation}\label{eq.AS.es}
  0 \to \Psi\sp{-1}(M) \to \Psi\sp{0}(M) \to \CI(S\sp*M) \to 0\,,
\end{equation}
where \m{S\sp*M} is the cosphere bundle of \m{M}, as before, (that is,
the set of vectors of length one of the cotangent space \m{T\sp*M} of
\m{M}). That is, in the exact sequence \eqref{eq.abs.es}, we have \m{I
  = \Psi\sp{-1}(M)}, \m{A = \Psi\sp{0}(M)}, and \m{\Symb:=A/I \simeq
  \CI(S\sp*M)}.

It is interesting to point out that both the AS index formula and the
Fredholm condition of Theorem \ref{thm.fc.as} are based on the exact
sequence \eqref{eq.AS.es}.  Of course, to actually determine the
index, one has to do additional work, but the information needed is
contained in the exact sequence.  This remains true for most of the
other index theorems. We continue with some remarks.

\begin{remark}\label{rem.index}
Let us see now how the exact sequence \eqref{eq.AS.es} and Theorem
\ref{thm.AS} are related. Recall the boundary map \m{\pa :
  K_1\sp{alg}(\Symb) \to K_0\sp{alg}(I)} in {\em algebraic}
\m{K}-theory associated to the exact sequence \eqref{eq.abs.es}, see
Equation \eqref{eq.b.map}, and let us assume that the ideal \m{I} of
that exact sequence consists of compact operators (i.e. \m{I \subset
  \maK}). We first consider the natural map
\begin{equation}\label{eq.Tr}
	Tr_* \, : \, K_0\sp{alg}(I) \, \to \, \ZZ \,,
\end{equation}
where the trace refers to the trace (or dimension) of a projection.
We have, of course, that \m{Tr_* : K_0\sp{alg}(\maK) = K_0(\maK) \to
  \ZZ} is the usual isomorphism.  Then \m{Tr_* \circ \pa} {\em
  computes the usual (Fredholm) index}, that is, we have the equality
of the morphisms
\begin{equation}\label{eq.ch.it}
	\ind \, = \, Tr_* \circ \pa \, :\ K_1\sp{alg}(\Symb) \,
        \stackrel{\pa}{\longrightarrow} \, K_0\sp{alg}(I) \,
        \stackrel{Tr_*}{-\!\!\!\longrightarrow} \, \CC \,.
\end{equation}
Indeed, if \m{P \in A} is invertible in $A/I = \Symb$, then, on the
one hand, \m{P} defines a class \m{[P] \in K_1(\Symb)}, and, on the
other hand, \m{P} is {\em Fredholm} and its Fredholm index is given by
\begin{equation}
 \ind(P) \, = \, Tr_* \circ \pa[P] \,.
\end{equation}
\end{remark}

We thus see that computing the index of a Fredholm
(pseudo)differential operator on \m{M} is equivalent to computing the
composite map \m{Tr_* \circ \pa \, :\ K_1(\Symb) \to \CC}. This
observation due to Connes is the starting point of the approach to
index theorems described in this paper.

\begin{remark} Let us discuss now shortly the role of the
Chern character in the AS index formula. First, let us recall that the
Chern character establishes an isomorphism \m{ch: K\sp*(M_1) \otimes
  \CC \to H\sp*(M_1) \otimes \CC} for any compact, smooth manifold
\m{M_1}. Moreover, in the case of the commutative algebra
\m{\CI(M_1)}, we have that \m{K_*(\CI(M_1)) \simeq K\sp{*}(M_1)} and
hence any group morphism \m{K_*(\CI(M_1)) \to \CC} factors through the
Chern character \m{ch : K\sp{*}(M_1) \to H\sp{*}(M_1) \otimes \CC}.
Returning to the AS index formula, we have that \m{\Symb = \CI(S^*M)}
and hence the index map $\ind = Tr_* \circ \pa : K_1(\Symb) \simeq
K\sp1(S\sp*M) \to \CC$ can be expressed solely in terms of the Chern
character. It is therefore possible to express the AS Index Formula
purely {\em in classical terms} (vector bundles and cohomology)
because the quotient \m{A/I := \Symb \simeq \CI(S^*M)} {\em is
  commutative.}
\end{remark}

\begin{remark}
Technically, one may have to replace the algebra \m{A} with \m{M_n(A)}
and take \m{P \in M_n(A)}, but this is not an issue since the
$K$-groups (both topological and algebraic) are invariant for the
replacement of \m{A} with its matrix algebras. However, the approach
to the index of elliptic (pseudo)differential operators using exact
sequences can be used to deal with operators \m{P} acting between
sections of {\em isomorphic} bundles. For non-compact manifolds (and
hence also for singular spaces), this is enough. For the AS index
formula, however, one may have to replace first \m{M} with \m{M \times
  S\sp{1}}. For this reason, in the case of the AS index formula,
Connes\rp\ approach using the tangent groupoid may be more convenient.
\end{remark}

\begin{remark}\label{rem.AS}
  Elliptic operators on a smooth, compact manifolds $M$ have certain
  properties that are similar to the properties to $\Gamma$-invariant
  elliptic operators on a covering space $\Gamma \to \tilde M \to M$
  (so here $\Gamma$ is the group of deck transformations of $\tilde M$
  and hence $\tilde M/\Gamma \simeq M$). The reason is that they
  correspond the the {\em same} Lie algebroid on $M$, namely $TM \to
  M$. The two frameworks correspond however to {\em different} Lie
  groupoids, and their analysis is consequently also quite
  different. In particular, if $\Gamma$ is non-trivial, one is lead to
  consider {\em von Neumann algebras} \cite{ConnesFol, Jones,
    Voiculescu}. This is related to the example of foliations
  discussed in the next subsection.
\end{remark}

\subsection{Cyclic homology and Connes' index formula for foliations}
\label{ssec.foliations}
The map \m{Tr_*} of the basic equation \m{\ind(P) = Tr_* \circ \pa[P]}
(recall Equation \eqref{eq.ch.it}, which is valid when \m{I \subset
  \maK}), is a particular instance of the pairing between {\em cyclic
  cohomology} and \m{K}-theory \cite{ConnesBook}. See also
\cite{ConnesNCG, Karoubi, LodayQuillen, Manin1, Tsygan}. This pairing
is even more important when \m{I \not \subset \maK}.  Let us explain
this.  Let us denote by \m{\HP^*(B)} the {\em periodic cyclic
  cohomology} groups of an algebra \m{B} (for topological algebras,
suitable topological versions of these groups have to be considered).

Let us look again at the general exact sequence of Equation
\eqref{eq.abs.es} and let \m{\phi} be a cyclic cocycle on \m{I}, that
is, \m{\phi \in \HP^0(I)}. A more general (higher) index theorem is
then to compute
\begin{equation*}
	\phi_* \circ \pa \, : \, K_1(\Symb) \, \to \, \CC \,.
\end{equation*}
It is known that \m{\phi_* \circ \pa = (\pa \phi)_*}, and hence the
map \m{\phi_* \circ \pa} is also given by a cyclic cocycle
\cite{NistorDocumenta}.

The map \m{\phi_*} and, in general, the approach to Index theory using
cyclic homology is especially useful for {\em foliations} for the
reasons that we are explaining now. We regard a foliation \m{(M,
  \maF)} of a smooth, compact manifold \m{M} as a sub-bundle \m{\maF
  \subset TM} that is {\em integrable} (that is, its space of smooth
sections, denoted \m{\Gamma(\maF)}, is closed under the Lie bracket).
Connes' construction of pseudodifferential operators along the leaves
of a foliation \cite{ConnesFol} then yields the exact sequence of
algebras
\begin{equation}\label{eq.Connes.es}
	0 \to \Psi_{\maF}^{-1}(M) \to \Psi_{\maF}^{0}(M)
        \stackrel{\sigma_0}{\longrightarrow} \CI(S^*\maF) \to 0 \,,
\end{equation}  
where \m{\sigma_0} is again the principal symbol, defined essentially
in the same manner as for the case of smooth manifolds. In fact, for
\m{\maF = TM}, with \m{M} a smooth, compact manifold, this exact
sequence reduces to the earlier exact sequence \eqref{eq.AS.es}. It
also yields a boundary (or index) map
\begin{equation*}
	\pa \, : \, K_1(\CI(S^*\maF)) = K^1(S^*\maF) \, \to \,
        K_0(\Psi_{\maF}^{-1}(M)) \simeq K_0(\CIc(\maF))\,,
\end{equation*} 
where \m{\CIc(\maF)} is the convolution algebra of the groupoid
associated to \m{\maF} and where the topological $K$-groups were used.
A main difficulty here is that there are few calculations of
\m{K_0(\Psi_{\maF}^{-1}(M))}.  These calculations are related to the
Baum-Connes conjecture, which is however known not to be true for
general foliations, see \cite{SkandalisFolBC} and the references
therein. See also \cite{PV1, PV2}. Another feature of the foliation
case is that, unlike our other examples to follow,
\m{\Psi_{\maF}^{-1}(M)} has no canonical proper ideals, so there are
no other index maps.

Unlike its \m{K}-theory, the cyclic homology of
\m{\Psi_{\maF}^{-1}(M)} is much better understood, in particular, it
contains as a direct summand the twisted cohomology of the classifying
space of the groupoid (graph) of the foliaton \cite{BrylinskiNistor}.
We thus have a large set of linearly independent cyclic cocycles and
hence many linear maps $ \phi_* : K_0(\CIc(\maF)) \to \CC \,,$ each of
which defines an index map
\begin{equation*}
	\phi_* \circ \pa : K_0(\CI(S^*\maF)) \to \CC \,.
\end{equation*} 

We will not pursue further the determination of \m{\phi_* \circ \pa },
but we note Connes' results in \cite{ConnesBook, ConnesFol,
  kordyukovFol}, the results of Benameur--Heitsch for Haeffliger
homology \cite{BenameurHeitsch}, the results of Connes--Skandalis
\cite{ConnesSkandalis}, and of myself for foliated bundles
\cite{NistorFol}. We also stress that in the case of foliations, it
is the ideal \m{I} that causes difficulties, whereas the quotient
\m{\Symb := A/I} is commutative and, hence, relatively easy to deal
with. The opposite will be the case in the following subsection.

\subsection{The Atiyah-Patodi-Singer index formula}
A related but different type of example is provided by the
Atiyah-Patodi-Singer (APS) index formulas \cite{APS1}. Let \m{\oM} be
an $n$-dimensional compact manifold with {\em smooth boundary} \m{\pa
  \oM}. By definition, this means that \m{\oM} is locally
diffeomorphic to an open subset of \m{[0, 1) \times \RR\sp{n-1}}. The
  transition functions for a manifold with boundary will be assumed be
  smooth. To \m{\oM} we attach the semi-infinite cylinder
\begin{equation*}
	\pa \oM \times (-\infty, 0\,] \,,
\end{equation*}
yielding a {\em manifold with cylindrical ends}. The metric is taken
to be a product metric \m{g = g_{\pa \oM} + dt^2} far on the end.
Kondratiev's transform $r = e^t$ then maps the cylindrical end to a
tubular neighborhood of the boundary, such that the cylindrical end
metric becomes $g = g_{\pa \oM} + (r^{-1}dr)^2$ near the boundary,
since \m{r^{-1}dr = e\sp{-t} d(e\sp{t}) = dt}. Thus, on \m{\oM}, we
consider two metrics: first, the initial, everywhere smooth metric
(including up to the boundary) and, second, the modified, singular
metric \m{g} that corresponds to the compactification of the
cylindrical end manifold.

We have thus obtained the simplest examples of a non-compact manifold,
that of a {\em manifold with cylindrical ends.} We will consider on
this non-compact manifold only differential operators with
coefficients that extend to smooth functions up to infinity (so, in
particular, they have {\em limits} at infinity).  Because of this, it
will be more convenient to work on \m{\oM} than on \m{\oM \cup \pa \oM
  \times (-\infty, 0]}. This is achieved by the Kondratiev transform
  \m{r = e\sp{t}}. The Kondratiev transform is such that \m{\pa_t}
  becomes \m{r \pa_r}. On \m{\oM}, we then take the coefficients to be
  smooth functions up to the boundary. Therefore, in local coordinates
  \m{(r, x') \in [0, \epsilon) \times \pa \oM} on the distinguished
    tubular neighborhood of \m{\pa \oM}, we obtain the following form
    for our differential operators (here \m{n = \dim(\oM)}):
\begin{equation}\label{eq.def.totchar}
	P \, = \, \sum_{|\alpha| \le m} \, a_{\alpha}(r, x') (r
        \pa_r)^{\alpha_1} \pa_{x'_2}^{\alpha_2} \ldots
        \pa_{x'_n}^{\alpha_n} \, = \, \sum_{|\alpha| \le m} \,
        a_{\alpha} (r \pa_r)^{\alpha_1} \pa^{\alpha'} \,.
\end{equation}
Operators of this form are called {\em totally characteristic
  differential operators.}

Away from the boundary, the definition of the principal symbol for a
totally characteristic differential operator is unchanged. However, in
the same local coordinates near the boundary as in Equation
\eqref{eq.def.totchar}, the {\em principal symbol} for the totally
characteristic differential operator of Equation
\eqref{eq.def.totchar} is
\begin{equation}\label{eq.def.princ}
	\sigma_m(P) \, := \, \sum_{|\alpha| = m} a_{\alpha} \xi^\alpha
        \,.
\end{equation}
Thus the principal symbols is {\em not} \m{\sum_{|\alpha| = m}
  a_{\alpha} r^{\alpha_1} \xi^\alpha} as one might first think! Other
than the fact that this definition of the principal symbol gives the
\dlp right results,\drp\ it can be motivated by considering the
original coordinates \m{(t, x\rp) \in (-\infty, 0] \times \pa \oM} on
  the cylindrical end.

The principal symbol is something that was encountered in the
classical case of the AS-index formula as well as in the case of
foliations, so it is not something significantly new in the case of
manifolds with cylindrical ends--even if in that case the definition
is slightly different. However, in the case of manifolds with
cylindrical ends, there is another significant new ingredient, which
will turn out to be both crucial and typical in the analysis on
singular spaces. This significant new ingredient is the {\em indicial
  family} of a totally characteristic differential operator. To define
and discuss the indicial family of a totally characteristic
differential operator, let \m{P} be as in Equation
\eqref{eq.def.totchar} and consider the same local coordinates near
the boundary as in that equation. The definition of the indicial
family \m{\widehat{P}} of \m{P} is then as follows (we
\underline{underline} the most significant new ingredients of the
definition):
\begin{equation}\label{eq.def.indicial}
	\widehat P(\tau) \, := \, \sum_{|\alpha| \le m} a_{\alpha}(\,
        \underline{0}\, , x') \underline{(\imath \tau)\,
        }^{\alpha_1}\, \pa^{\alpha'} \,.
\end{equation}
Note that \m{\widehat{P}(\tau)} is a family of differential operators
on \m{\pa \oM} that depends on the coefficients of \m{P} only through
their restrictions to the boundary. Moreover, we see that the indicial
family \m{\widehat{P}(\tau)} of \m{P} is the Fourier transform of the
operator
\begin{equation}\label{eq.def.maI}
	\maI (P) \, := \, \sum_{|\alpha| \le m} a_{\alpha}(\,
        \underline{0}\, , x') \underline{\pa_t\, }^{\alpha_1}\,
        \pa^{\alpha'} \,,
\end{equation}
which is a translation invariant operator on \m{\pa \oM \times \RR}.
The operator \m{\maI(P)} is called the {\em indicial operator} of
\m{P} \cite{Lesch, Schulze}. Note that we have denoted $\pa^{\alpha} =
\pa_t^{\alpha_1} \pa^{\alpha'}$.

We are interested in Fredholm conditions for totally characteristic
differential operators, so let us introduce the last ingredient for
the Fredholm conditions. Let us endow \m{M := \oM \smallsetminus \pa
  \oM} with a cylindrical end metric. Since the cylindrical end metric
is complete, the Laplacian \m{\Delta} is self-adjoint, and hence we
can define the Sobolev space \m{H^s(M)} as the domain of
\m{\Delta^{s/2}}, that is, \m{H^s(M) := \maD(\Delta^{s/2})}, which
turns out to be independent of the choice of the cylindrical end
metric. (We consider all differential operators to be defined on their
minimal domain, unless otherwise mentioned, and thus, in particular,
they give rise to closed, densely defined operators.) We then have a
characterization of Fredholm totally characteristic differential
operators similar to the compact case (the differences to the compact
case are \underline{underlined}).

\begin{theorem}\label{thm.fc.aps}
Assume \m{M} has cylindrical ends and \m{P} is a totally
characteristic differential operator of order $m$ acting between the
sections of the bundles \m{E} and \m{F}.  Then, for any fixed $s \in
\RR$, we have that
\begin{multline*}
  P : H^s(M; E) \to H^{s-m}(M; F) \mbox{ is Fredholm } \\
  \Leftrightarrow \ P \mbox{ is elliptic } \underline{\mbox{and }
    \widehat{P}(\tau) \mbox{ is invertible } \forall \tau \in \RR }\,.
\end{multline*}
\end{theorem}

This result has a long history and related theorems are due to many
people, too many to mention them all here. Nevertheless, one has to
mention the pioneering work of Lockhart-Owen on differential operators
\cite{LockhartOwen} and the work of Melrose-Mendoza for totally
characteristic pseudodifferential operators \cite{MelroseMendoza}. A
closely related theorem for differential operators and domains with
conical points has appeared in a landmark paper by Kondratiev in 1967
\cite{Kondratiev67}. Other important results in this direction were
obtained by Mazya \cite{KMR} and Schrohe--Schultze
\cite{SchroheFrechet, SchSch1}. See the books of Schulze
\cite{SchulzeBook91}, Lesch \cite{Lesch}, and Plamenevski\u{\i} \cite{plamenevskiiBook}
for introductions and more
information on the topics and results of this subsection.

One can easily show that \m{\maI(P)} is invertible if, and only if,
\m{\hat P(\tau)} is invertible for all \m{\tau \in \RR}.  Thus the
Fredholmness criterion of Theorem \ref{thm.fc.aps} can also be given
the following formulation that is closer to our more general result of
Theorem \ref{thm.fc.Lie}.

\begin{theorem}\label{thm.fc.aps2}
Let \m{M} and \m{P} be as in Theorem \ref{thm.fc.aps} and $s \in
\RR$. Then
\begin{multline*}
  P : H^s(M; E) \to H^{s-m}(M; F) \mbox{ is Fredholm } \\
  \Leftrightarrow \ P \mbox{ is elliptic } \underline{\mbox{and }
    \maI(P) \mbox{ is invertible.}}
\end{multline*}
\end{theorem}

Let us consider now a totally characteristic, twisted Dirac operator
\m{P}. In case \m{P} is Fredholm, its Fredholm index is given by the
Atiyah-Patodi-Singer (APS) formula \cite{APS1}, which expresses
\m{\ind(P)} as the sum of two terms:
\begin{enumerate}[(i)]
\item the integral over \m{\oM} of an explicit form, which is a local
  term that depends only on the principal symbol of the operator
  \m{P}, as in the case of the AS formula, and
\item a boundary contribution that depends only on the indicial family
  \m{\widehat{P}(\tau),} namely the $\eta$-invariant, which is this
  time a {\em non-local} invariant. It can be expressed in terms of
  \m{\maI(P)} \cite{MelroseNistor}.
\end{enumerate}
We thus see that even to be able to formulate the APS-index formula,
we need to know which totally characteristic operators will be
Fredholm. Moreover, the ingredients needed to compute the index of
such an operator \m{P} (that is, its principal symbol $\sigma_m(P)$
and the indicial operator \m{\maI(P)}) are exactly the ingredients
needed to decide that the given operator \m{P} is Fredholm. See
\cite{Bismut, LauterMoroianu2, MelroseNistor,
  CarilloLescureMonthubert} for further results.

Let us now introduce the exact sequence of the APS index formula.
First of all, \m{A := \Psi^0_b(\oM)} is the algebra of {\em totally
  characteristic} pseudodifferential operators on \m{\oM}. One of its
main properties is that the differential operators in \m{A} are
exactly the totally characteristic differential operators. See
\cite{Lesch, SchulzeBook91} for a definition of
\m{\Psi^\infty_b(\oM)}. A definition using groupoids (of a slightly
different algebra) will be given in Subsection \ref{ssec.psdo} in a
more general setting.  Let \m{r} be a defining function of the
boundary \m{\pa \oM} of \m{\oM}, as before.  Next, the ideal is \dm{I
  \,:= \, r \Psi^{-1}_b(\oM) \, = \, \Psi^0_b(\oM) \cap \maK\,.} Then
the symbol algebra \m{\Symb := A/I} is the fibered product
\begin{equation}\label{eq.def.saps}
  \Symb = \CI(S^*\oM) \oplus_{\pa} \Psi^0(\pa \oM \times \RR)^\RR \,,
\end{equation}
more precisely, \m{\Symb} consists of pairs \m{(f, Q)} such that the
principal symbol of the \m{\RR} invariant pseudodifferential operator
\m{Q} matches the restriction of \m{f \in \CI(S^*\oM)} to the
boundary. Recalling the definition of \m{\maI} in Equation
\eqref{eq.def.maI} (and extending it to totally characteristic
pseudodifferential operators), we obtain the exact sequence
\begin{equation}\label{eq.APS.es}
	0 \to r \Psi^{-1}_{b}(\oM) \to \Psi^{0}_{b}(\oM)
        \stackrel{\sigma_0 \oplus \maI}{-\!\!\!\longrightarrow}
        \CI(S^*\oM) \oplus_{\pa} \Psi^0(\pa \oM \times \RR)^\RR \to 0
        \,.
\end{equation}  
(The exact sequence \m{0 \to \Psi^{-1}_{b}(\oM) \to \Psi^{0}_{b}(\oM)
  \stackrel{\sigma_0}{\longrightarrow} \CI(S^*\oM) \to 0 } is, by
contrast, less interesting.)

The exact sequence \eqref{eq.APS.es} in particular gives that \m{P} is
Fredholm if, and only if, the pair \m{(\sigma_0(P), \maI(P)) \in \Symb
  := \CI(S^*M) \oplus_{\pa} \Psi^0(\pa \oM \times \RR)^\RR} is
invertible, which, in turn, is true if, and only if, \m{P} is elliptic
and \m{\maI(P)} is invertible. Thus the exact sequence
\eqref{eq.APS.es} implies Theorem \ref{thm.fc.aps2}.

As before, composing \m{\pa : K_1(\Symb) \to K_0(I)}, \m{I =
  r\Psi^{-1}(\oM))}, with the trace map $Tr_*$ gives us the Fredholm
index
\begin{equation*}
  \ind \, = \, Tr_* \circ \pa \, : \, K_1(\Symb) \, \to \, \CC \,.
\end{equation*}
Since \m{Tr_* \circ \pa = (\pa Tr)_*} \cite{ConnesNCG} (see
\cite{NistorDocumenta} for the case when $Tr$ is replaced by a general
cyclic cocycle), we see that the APS index formula is also equivalent
to the calculation of the class of the cyclic cocycle \m{\pa Tr \in
  \HP^1(\Symb)}.  This was the approach undertaken in
\cite{MelroseNistor, MoroianuNistor}.

\begin{remark}
It is important to stress here first the role of cyclic homology,
which is to define natural morphisms \m{K_1(\Symb) \to \CC}, morphisms
that are otherwise difficult to come by. Also, it is important to
stress that it is the noncommutativity of the algebra of symbols
\m{\Symb} that explains the fact that the APS index formula is
non-local.
\end{remark}

We need to insist of the fact that in the case of the APS framework,
it is the symbol algebra \m{\Symb := A/I} that causes difficulties, in
large part because it is non-commutative (so the classical Chern
character is not defined), whereas the ideal \m{I \subset \maK} is
easy to deal with. This is an opposite situation to the one
encountered for foliations.  It is for this reason that the foliation
framework and the APS framework extend the AS framework in different
directions.

The approach to Index theory explained in this last subsection extends
to more complicated singular spaces, and this has provided the author
of this presentation the motivation to study analysis on singular
spaces.

\section{Motivation: Degeneration and singularity\label{sec2}}

The totally characteristic differential operators studied in the
previous subsection appear not only in index problems, but actually
arise in many practical applications. We shall now examine how the
totally characteristic differential operators and other related
operators appear in practice. In a nut-shell, these operators can be
used to model degenerations and singularities. In this section, we
introduce several examples. We begin with the ones related to the APS
index theorem (the totally characteristic ones, called \dlp rank
one\drp\ by analogy with locally symmetric spaces) and then we
continue with other examples. Again, no results in this section are
new.

\subsection{APS-type examples: rank one}

Let us denote by \m{\rho} the distance to the origin in \m{\RR\sp{d}}.
Here is a list of examples of totally characteristic operators.

\begin{example} In our three examples below, the first
one is a true totally characteristic operator, whereas the other two
require us to remove the factor \m{\rho\sp{-2}} first.

\begin{enumerate}
\item The elliptic generator \m{L} of the Black-Scholes 
 PDE \m{\pa_t - L} \cite{Shreve}
\begin{equation*}
   Lu \, := \, \frac{\sigma^2}{2} x^2\pa_x^2 u + r x \pa_x
  u - ru \,.
\end{equation*}
\item The Laplacian in polar coordinates
  \m{(\rho, \theta)}
\begin{equation*}
  \Delta u \, = \, \rho^{-2} \big( \rho^2 \pa_{\rho}^2 u + {\rho \pa_{\rho}}
  u + \pa_\theta^2u \big).
\end{equation*}
\item The Schr\"{o}dinger operator in spherical coordinates \m{(\rho,
  x\rp)}, \m{x\rp \in S\sp{2}},
\begin{equation*}
  -(\Delta + \frac{Z}{\rho}) u \, = \, -\rho^{-2} \big( \rho^2
  \pa_{\rho}^2 u + 2 {\rho \pa_{\rho}} u + \Delta_{S^2}u + Z \rho u
  \big).
\end{equation*}
\end{enumerate}
\end{example}

A similar expansion is valid for elliptic operators in generalized
spherical coordinates in arbitrary dimensions and was used by
Kondratiev in \cite{Kondratiev67} to study domains with conical
points. Kondratiev\rp s paper is widely used since it provides the
needed analysis facts to deal with polygonal domains, the main testing
ground for numerical methods.

\subsection{Manifolds with corners}
For more complicated examples we will need manifolds with corners.
Recall that \m{\oM} is a {\em manifold with corners} if, and only if,
\m{\oM} is locally diffeomorphic to an open subset of \m{[0,
    1)\sp{n}}. The transition functions of \m{\oM} are supposed to be
  smooth, as in the case of manifolds with smooth boundary. A manifold
  with boundary is a particular case of a manifold with corners, but
  we agree in this paper that a {\em smooth manifold does not have
    boundary} (or corners), since we regard the corners (or boundary)
  as some sort of singularity.

A point $p \in \oM$ is called of {\em depth} $k$ if it has a
neighborhood $V_p$ diffeomorphic to $[0, 1)^{k} \times (-1, 1)^{n-k}$
  by a diffeomorphism $\phi_p : V_p \to [0, 1)^{k} \times (-1,
    1)^{n-k}$ mapping $p$ to the origin:\ $\phi_p(p) = 0$. A connected
    component $F$ of the set of points of depth $k$ will be called an
    open face (of codimension $k$) of $\oM$. The set of points of
    depth $0$ of $\oM$ is called the {\em interior} of $\oM$, is
    denoted $M$, and its connected components are also considered to
    be an open faces of $\oM$. The closure in $\oM$ of an open face
    $F$ of $\oM$ will be called a {\em closed} face of $\oM$. A closed
    face of $\oM$ may not be a manifold with corners in its own. The
    union of the proper faces of \m{\oM} is denoted by \m{\pa \oM} and
    is called the {\em boundary} of \m{\oM}. Thus \m{M:= \oM
      \smallsetminus \pa \oM}.

The following set of vector fields will be useful when defining Lie
manifolds:
\begin{equation}\label{eq.def.maVb}
  \maV_{b} \, := \, \{\, X \in \Gamma(\oM; T\oM), \ X \, \mbox{ is
    {\em tangent} to all boundary faces of } \, \oM \, \} \,.
\end{equation}
Let us notice that in the case of manifolds with boundary, the totally
characteristic differential operators on \m{\oM}, see Equation
\eqref{eq.def.totchar}, are generated by \m{\CI(\oM)} and the vector
fields \m{X \in \maV_b}.

\subsection{Higher rank examples}
We now continue with more complicated examples, which we call \dlp
higher rank\drp\ examples, again by analogy with locally symmetric
spaces. In general, the natural domains for these higher rank examples
will be manifolds with corners.

\begin{example}\label{ex.edge}
There are no \dlp higher rank\drp\ example in dimension one, so we
begin with an example in dimension two.
 
\begin{enumerate}
 \item The simplest non-trivial example is the Laplacian
\begin{equation*}
  \Delta_{\HH} = y\sp{2}( \pa_x\sp{2} + \pa_y \sp{2})
\end{equation*}
on the hyperbolic plane \m{\HH = \RR \times [0, \infty)}, whose
metric is \m{y\sp{-2}(dx\sp2 + dy\sp2)}. 

\item The Laplacian on the hyperbolic plane is closely related to the
  SABR Partial Differential Equation (PDE) due to Lesniewsky and
  collaborators \cite{SABR}. The SABR PDE is also a parabolic PDE
  \m{\pa_t - L} associated to a stochastic differential equation, with
\begin{equation*}
  2L \, := \, y^2\big( x^2 \partial_x^2 + 2\rho\nu x \pa_x \pa_y +
  \nu^2 \pa_y^2 \big) \, ,
\end{equation*}
with \m{\rho} and \m{\nu} real parameters. Stochastic differential
equations provide many interesting and non-trivial examples of
degenerate parabolic PDEs that can be treated using Lie manifolds.

\item A related example is that of the Laplacian in cylindrical
  coordinates \m{(\rho, \theta, z)} in three dimensions:
\begin{equation}\label{eq.def.edge}
   \Delta u \, = \, \rho^{-2} \big( (\rho \pa_{\rho})^2 u +
   \pa_\theta^2u + (\rho \pa_z)^2 \big) \, .
\end{equation}
Ignoring the factor \m{\rho^{-2}}, which amounts to a conformal change
of metric, we see that our differential operator (that is,
$\rho\sp{2}\Delta$) is generated by the vector fields
\begin{equation*}
	\rho \pa_\rho, \ \pa_\theta , \ \mbox{ and } \ \, \rho \pa_z \,,
\end{equation*}
and that the linear span of these vector fields is a {\em Lie
  algebra}. The resulting partial differential operators are usually
called {\em edge differential operators}. This example can be used to
treat the behavior near edges of polyhedral domains of elliptic
PDEs. This behavior is more difficult to treat than the behavior near
vertices.  For a boundary value problems in a three dimensions wedge
of dihedral angle $\alpha$, the natural domain is \m{[0, \alpha]
  \times [0, \infty) \times \RR}, a manifold with corners of
  codimension two.
\end{enumerate}

\end{example}

We thus again see that Lie algebras of vector fields are one of the
main ingredients in the definition of the differential operators that
we are interested in. More related examples will be provided below as
examples of Lie manifolds.

Degenerate elliptic equations have many applications in Numerical
Analysis, see \cite{BNZ3D2, CDN12, DaugeBook, Hengguang, HMN}, for
example.

\section{Lie manifolds: definition and geometry\label{sec3}}

Motivated by the previous two sections, we now give the definition of
a Lie manifold largely following \cite{aln1}. We also introduce a
slightly more general class of manifolds than in \cite{aln1} by
allowing the manifold with corners appearing in the definition to be
noncompact.  We also slightly simplify the definition of a Lie
manifold based on a comment of Skandalis. We thus define our Lie
manifolds using Lie algebroids and then we recover the usual
definition in terms of Lie algebras of vector fields.  I have tried to
make this section as self-contained as possible, thus including most
of the proofs, some of which are new.

\subsection{Lie algebroids and Lie manifolds}
\label{ssec.liealg}
We have found it convenient to introduce Lie manifolds and \dlp open
manifolds with a Lie structure at infinity\drp\ in terms of Lie
algebroids, which we will recall shortly. First, recall that we use
the following notation, if \m{E \to X} is a smooth vector bundle, we
denote by \m{\Gamma(X; E)} (respectively, by \m{\Gamma_c(X;E)}) the
space of smooth (respectively, smooth, compactly supported) sections
of \m{E}. Sometimes, when no confusion can arise, we simply write
\m{\Gamma(E)}, or, respectively, \m{\Gamma_c(E)}. We now introduce Lie
algebroids. Lie algebroids were introduced by Pradines
\cite{Pradines}. See also \cite{HigginsMackenzie1, HigginsMackenzie2}
for some basic results on Lie algebroids and Lie groupoids.  We refer
to \cite{MackenzieBook1, MoerdijkFolBook} for further material and
references to Lie algebroids and groupoids.

\begin{definition} 
A {\em Lie algebroid} \m{A \to \oM} is a real vector bundle over a
manifold with corners \m{\oM} together with a {\em Lie algebra}
structure on \m{\Gamma(\oM;A)} (with bracket \m{[\ , \ ]}) and a
vector bundle map \m{\varrho: A \rightarrow T\oM}, called {\em
  anchor}, such that the induced map \m{\varrho_* : \Gamma(\oM; A) \to
  \Gamma(\oM; T\oM)} satisfies the following two conditions:
\begin{enumerate}[(i)]
\item \m{\varrho_*([X,Y]) = [\varrho_*(X),\varrho_*(Y)]} and
\item \m{[X, fY] = f[X,Y] + (\varrho_*(X) f)Y}, for all \m{X, Y \in
  \Gamma(\oM; A)} and \m{f \in \CI(\oM)}.
\end{enumerate}
\end{definition}

For further reference, let us recall here the {\em isotropy} of a Lie
algebroid.

\begin{definition}\label{def.isotropy}
Let \m{\varrho : A \to T\oM} be a Lie algebroid on \m{\oM} with anchor
\m{\varrho}. Then the kernel \m{\ker (\varrho_x : A_x \to T_x \oM)} of
the anchor is the {\em isotropy} of \m{A} at \m{x \in \oM}.
\end{definition}

The isotropy at any point can be shown to be a Lie algebra. See
\cite{ASkandalis1} for generalizations. Recall that we
denote by \m{\pa \oM} the boundary \m{\oM}, that is, the union of its
proper faces, and by \m{M:= \oM \smallsetminus \pa \oM} its interior.

\begin{definition} \label{def.LieManifold}
A pair $(\oM, A)$ consisting of a manifold with corners $\oM$ and a
Lie algebroid $A \to \oM$ is called an {\em open manifold with a Lie
  structure at infinity} if its anchor \m{\varrho: A \to T\oM}
satisfies the following properties:
\begin{enumerate}[(i)]
\item \m{\varrho: A_x \to T_x\oM} is an isomorphism for all \m{x \in M
  := \oM \smallsetminus \pa M} and
\item \m{\maV := \varrho_*(\Gamma(\oM; A)) \subset \maV_b}.
\end{enumerate}
If \m{\oM} is compact, then the pair \m{(\oM, A)} will be called a
{\em Lie manifold}.
\end{definition}

Condition (ii) means that the Lie algebra of vector fields \m{\maV :=
  \varrho_*(\Gamma(A))} consists of vector fields tangent to all faces
of \m{\oM}. One of the main reason for introducing open manifolds with
a Lie structure at infinity is in order to be able to localize Lie
manifolds. Thus, if \m{(\oM, A)} is a Lie manifold and \m{V \subset
  \oM} is an open subset, then \m{(V, A\vert_{V})} will not be a Lie
manifold, in general, but will be an open manifold with a Lie
structure at infinity.  Thus the Lie manifolds are exactly the compact
open manifold with a Lie structure at infinity.  Lie manifolds were
introduced in \cite{aln1}.

By extension, $\oM$ and $M := \oM \smallsetminus \pa \oM$ in
Definition \ref{def.LieManifold} will also be called open manifolds
with a Lie structure at infinity.  We shall write \m{\Gamma(A)}
instead of \m{\Gamma(\oM; A)} when no confusion can arise, also, we
shall usually write \m{\Gamma(A)} instead of \m{\varrho_*(\Gamma(A))}.
We have the following \dlp trivial\drp\ example.

\begin{example}
The ``example zero'' of a Lie manifold is that of a smooth, compact
manifold \m{M = \oM} (no boundary or corners) and is obtained by
taking \m{A = TM}, thus \m{\maV = \Gamma(TM) = \Gamma(M; TM) =
  \maV_b}. Then \m{(M, A)} is a (trivial) example of a Lie
manifold. This example of a Lie manifold provides the framework for
the AS Index Theorem. Similarly, every smooth manifold \m{M} is an
open manifold with a Lie structure at infinity by taking \m{\oM = M}
and \m{A = TM}.
\end{example}

\begin{example} Let \m{\oM} be a manifold with corners
such that its interior \m{M := \oM \smallsetminus \pa \oM} identifies
with the quotient of a {\em Lie group} \m{G} by a discrete subgroup
\m{\Gamma} and the action of \m{G} on \m{G/\Gamma} by left
multiplication extends to an action of \m{G} on \m{\oM}. Let \m{\mfk
  g} be the Lie algebra of \m{G}. Then \m{A := \oM \times \mfk g} with
anchor given by the infinitesimal action of \m{G} is naturally a Lie
algebroid. Note that the action of the Lie algebra \m{\mfk g}
preserves the structure of faces of \m{\oM} and hence
\m{\varrho_*(\Gamma(A)) \subset \maV_b}. We call the corresponding
manifold with a Lie structure at infinity {\em a group enlargement.}
The simplest example is that of \m{G = M = \RR_+\sp{*}} acting on
\m{[0, \infty]}. Many interesting Lie manifolds arising in practice
are, locally, group enlargements, see for instance
\cite{GeorgescuNistor2, GeorgescuNistor1} for some examples coming
from quantum mechanics.
\end{example}

Let \m{(\oM, A)} be an open manifold with a Lie structure at infinity.
In applications, it is easier to work with the vector fields \m{\maV
  := \varrho_*(\Gamma(A))}, associated to a Lie manifold $(M, A)$,
than with the Lie algebroid $A \to M$.  We shall then use the
following alternative definition of Lie manifolds.

\begin{proposition} \label{prop.LieManifold}
Let us consider a pair \m{(\oM, \maV)} consisting of a compact
manifold with corners \m{\oM} and a subspace \m{\maV \subset
  \Gamma(\oM; T\oM)} of vector fields on \m{\oM} that satisfy:
\begin{enumerate}[(i)]
\item \m{\maV} is closed under the Lie bracket \m{[\;,\;]};
\item \m{\Gamma_c(M; TM) \subset \maV \subset \maV_b};
\item \m{\CI(\oM) \maV = \maV} and \m{\maV} is a finitely-generated
  \m{\CI(\oM)}--module;
\item \m{\maV} is projective (as a \m{\CI(\oM)}--module).
\end{enumerate} 
Then there exists a Lie manifold \m{(\oM, A)} with anchor \m{\varrho}
such that \m{\varrho_*(\Gamma(\oM; A)) = \maV}.

Conversely, if \m{(\oM, A)} is a Lie manifold, then \m{\maV :=
  \varrho_*(\Gamma(\oM; A))} satisfies the conditions (i)--(iv) above.
\end{proposition}

\begin{proof}
Let \m{(\oM, \maV)} be as in the statement. Since \m{\maV} is a
finitely generated, projective \m{\CI(\oM)}--module, the Serre--Swan
Theorem implies then that there exists a finite dimensional vector
bundle \m{A_\maV \to \oM}, uniquely defined up to isomorphism, such
that
\begin{equation}\label{eq.def.A}
  \maV \, \simeq \, \Gamma(\oM; A_\maV) \,,
\end{equation}
as \m{\CI(\oM)}--modules.  Let \m{I_x := \{ \phi \in \CI(\oM), \,
  \phi(x) = 0\}} be the maximal ideal corresponding to \m{x \in \oM}.
The fibers \m{(A_{\maV})_{x}}, \m{x \in \oM}, of the vector bundle
\m{A_{\maV} \to \oM} are given by \m{(A_{\maV})_{x} = \maV/I_{x}
  \maV}. Since \m{\Gamma(\oM; A_\maV) \simeq \maV \subset \Gamma(\oM;
  T\oM)}, we automatically obtain for each \m{x \in \oM} a map
\begin{equation*}
  (A_{\maV})_{x} \, := \, \maV/I_{x} \maV \, \rightarrow\, \Gamma(\oM;
  T\oM)/I_{x}\Gamma(\oM; T\oM) \, = \, T_x \oM\,.
\end{equation*}
These maps piece together to yield a bundle map (anchor) \m{\varrho :
  A_{\maV} \to T\oM} that makes \m{A_{\maV} \to \oM} a Lie
algebroid. The anchor map \m{\varrho} is an isomorphism over the
interior \m{M} of \m{\oM} since \m{\Gamma_c(M; TM) \subset \maV},
which is part of Assumption (ii). Since \m{\maV \subset \maV_b}, again
by Assumption (ii), we obtain that \m{(\oM, A_\maV)} is indeed a Lie
manifold.

Conversely, let \m{(\oM, A)} be a Lie manifold with anchor \m{\varrho
  : A \to T\oM}.  We need to check that \m{\maV :=
  \varrho_*(\Gamma(\oM; A))} satisfies conditions (i)--(iv) of the
statement.  Indeed, \m{\maV := \varrho_*(\Gamma(\oM; A))} is a Lie
algebra because \m{\Gamma(\oM; A)} is a Lie algebra and \m{\varrho_* :
  \Gamma(\oM; A) \to \Gamma(\oM; T\oM)} is an injective Lie algebra
morphism. So Condition (i) is satisfied.  To check the second
conditions, we notice that Definition \ref{def.LieManifold}(i)
(isomorphism over the interior) gives that \m{\Gamma_c(M; T\oM)
  \subset \maV}. Since we have by assumption \m{\maV \subset \maV_b},
we see that Condition (ii) is also satisfied. Finally, Conditions
(iii) and (iv) are satisfied since the space of smooth sections of a
finite dimensional vector bundle defines a projective module over the
algebra of smooth functions on the base, again by the Serre-Swan
theorem.
\end{proof}

Let \m{(\oM, \maV)} as in the statement of the above proposition,
Proposition \ref{prop.LieManifold}. We call \m{\maV} its {\em
  structural Lie algebra of vector fields} and we call the Lie
algebroid \m{A_\maV \to \oM} introduced in Equation \eqref{eq.def.A}
the {\em the Lie algebroid associated to} \m{(\oM, \maV)}.  The
alternative characterization of Lie manifolds in Proposition
\ref{prop.LieManifold} is the one that will be used in our examples.

\begin{remark}\label{rem.localBasis}
It is worthwhile pointing out that the condition that \m{\maV} be a
finitely generated, projective \m{\CI(\oM)}--module in Proposition
\ref{prop.LieManifold} together with the fact that the anchor
\m{\varrho} is an isomorphism over the interior of \m{\oM} are
equivalent to the following condition, where \m{n = \dim(\oM)}:
\begin{quotation}
  For every point \m{p \in \oM}, there exist a neighborhood \m{V_p} of
  \m{p} in \m{\oM} and \m{n}-vector fields \m{X_1, X_2, \ldots, X_n
    \in \maV} such that, for any vector field \m{Y \in \maV}, there
  exist smooth functions \m{\phi_1, \phi_2, \ldots, \phi_n \in
    \CI(\oM)} such that
\begin{equation}
  Y = \phi_1 X_1 + \phi_2 X_2 + \ldots + \phi_n X_n \text{ on } V_p
  \,,\quad \text{with } \ \phi_i\vert_{V_p} \mbox{ uniquely
    determined.}
\end{equation}
\end{quotation}
The vector fields \m{X_1, X_2, \ldots, X_n} are then called a {\em
  local basis} of \m{\maV} on $V_p$. (This is the analog in our case
of the well known fact from commutative algebra that a module is
projective if, and only if, it is locally free.)
\end{remark}

In the next example, we shall need the defining functions of a \dlp
hyperface.\drp\ A {\em hyperface} is a proper face \m{H \subset \oM}
of maximal dimension (dimension \m{\dim(H) = \dim(M) - 1}). Recall
that a {\em defining function} of a hyperface \m{H} of \m{\oM} is a
function \m{x} such that \m{H = \{x = 0 \}} and \m{dx \neq 0} on
\m{H}.  The hyperface \m{H \subset \oM} is called {\em embedded} if it
has a defining function. The existence of a defining function is a
global property, because locally one can always find defining
functions, a fact that will be needed in the example below.

The simplest example of a non-compact Lie manifold is that of a
manifold with cylindrical ends. The following example generalizes this
example to the higher rank case. It is a basic example to which we
will come back later.

\begin{example} \label{ex.one}
Let \m{\oM} a compact manifold with corners and \m{\maV =
  \maV_{b}}. Let us check that \m{(\oM, \maV_b)} is a Lie manifold.
We shall use Proposition \ref{prop.LieManifold}. Condition (i) is
easily verified since the Lie bracket of two vector fields tangent to
a submanifold is again tangent to that submanifold. Condition (ii) in
the Proposition \ref{prop.LieManifold} is even easier since, by
definition, vector fields that are zero near the boundary \m{\pa \oM}
are contained in \m{\maV_{b}}. Clearly, \m{\maV} is a \m{\CI(\oM)}
module. The only non-trivial fact to check is that \m{\maV} is
finitely generated and projective as an \m{\CI(\oM)} module.  This is
actually the only fact that we still need to check.  To verify it, let
us fix a corner point \m{p} of codimension \m{k} (that is, \m{p}
belongs to an open face \m{F} of codimension \m{k}). Then, in a
neighborhood of \m{p}, we can find \m{k} defining functions \m{r_1,
  r_2, \ldots, r_k} of the hyperfaces containing \m{p} such that a
local basis of \m{\maV} around \m{p} (see Remark \ref{rem.localBasis})
is given by
\begin{equation}\label{eq.local.basis.one}
  r_1 \pa_{r_1}, \, r_2 \pa_{r_2}, \, \ldots,\, r_k \pa_{r_k},\,
  \pa_{y_{k+1}}, \, \ldots , \, \pa_{y_n} \,,
\end{equation}
where \m{y_{k+1}, \ldots, y_n} are local coordinates on the open face
\m{F} of dimension \m{k} containing \m{p}, so that \m{(r_1, r_2,
  \ldots, r_k , y_{k+1}, \ldots, y_{n})} provide a local coordinate
system in a neighborhood of \m{p} in \m{\oM}. If \m{\oM} has a smooth
boundary, then \m{\maV_{b}} generates the totally characteristic
differential operators, which were introduced in Equation
\eqref{eq.def.totchar}, and hence this example corresponds to a
manifold with cylindrical ends.  In fact, we will see that the natural
Riemannian metric of a manifold with (asymptotically) cylindrical
ends. This example was studied also by Debord and Lescure
\cite{DebordLescure2, DebordLescure}, Melrose and Piazza
\cite{MelrosePiazza}, Monthubert \cite{Monthubert}, Schulze
\cite{SchulzeBook91}, and many others.
\end{example}

By \cite{aln1}, every vector field \m{X \in \maV} that has compact
support in \m{\oM} gives rise to a one parameter group of
diffeomorphisms \m{\exp(tX) : \oM \to \oM}, \m{t \in \RR}. We denote
by \m{\exp(\maV)} the subgroup of diffeomorphisms generated by all
\m{\exp(X)} with \m{X \in \maV} and compact support in \m{\oM}. The
results in \cite{ASkandalis1} show that \m{\exp(\maV)} acts by Lie
automorphisms of \m{\maV} (the condition (iv) of Proposition
\ref{prop.LieManifold} that \m{\maV} be a projective module is not
necessary). Also, it would be interesting to see how the groups
\m{\exp(\maV)} fit into the general theory of infinite dimensional Lie
groups \cite{Neeb}.

\subsection{The metric on Lie manifolds}

As seen from the example of manifolds with cylindrical ends, Lie
manifolds have an intrinsic geometry. We now discuss some results in
this direction following \cite{aln1} and we extend them to open
manifolds with a Lie structure at infinity (this extension is
straightforward, but needed). Thus, from now on, \m{(\oM, A)} will be
an open manifold with a Lie structure at infinity.  (Thus we will not
assume \m{(\oM, A)} to be a Lie manifold, unless explicitly stated.)

\begin{definition}\label{def.compatible}
Let \m{(\oM, A)} be an open manifold with a Lie structure at infinity.
A metric on \m{TM} is called {\em compatible} (with the structure at
infinity) if it extends to a metric on \m{A \to \oM}.
\end{definition}

We shall need the following lemma.

\begin{lemma}\label{lemma.distance}
Let \m{(\oM, A)} be an open manifold with a Lie structure at infinity
with compatible metric \m{g}. Assume \m{\oM} to be paracompact. Then
there exists a smooth metric \m{h} on \m{T\oM} such that \m{h \le g}.
\end{lemma}

\begin{proof}
Let us choose an arbitrary metric \m{h_0} on \m{\oM} (or, more
precisely, on \m{T\oM}). For each \m{p \in \oM}, let \m{U_p \subset
  V_p} be open neighborhoods of \m{p} in \m{\oM} such that \m{V_p} has
compact closure and contains the closure of \m{U_p}. Since \m{V_p} has
compact closure and the anchor map \m{\varrho} is continuous, we
obtain that there exists \m{M_p > 0} such that \m{h_0(\xi) \le M_p
  g(\xi)} for every \m{\xi \in A\vert_{V_p}}. Let us choose $I \subset
\pa \oM$ such that \m{ \{U_p, p \in I\} } is a locally finite covering
of the boundary \m{\pa \oM}. Let \m{V_0} be the complement of
\m{\cup_{p \in I} \overline{U_p}} and let \m{(\phi_p)_{p \in I \cup
    \{0\}}} be a smooth, locally finite partition of unity on \m{\oM}
subordinated to the covering \m{(V_p)_{p \in I \cup \{0\}}}. Then, if
we define
\begin{equation*}
  h \, = \, \sum_{p \in I} \phi_p M_{p}\sp{-1} h_0 + \phi_0 g \,,
\end{equation*}
the metric \m{h} will satisfy \m{h \le g} everywhere, as desired.
\end{proof}

Let us fix from now on a metric \m{g} on \m{A}, which restricts to a
compatible Riemannian metric on \m{M}. The inner product of two
vectors (or vector fields) \m{X, Y \in \Gamma(M; TM)} in this metric
will be denoted \m{(X, Y) \in \CI(M)} and the associated volume form
\m{d\vol_g}. Of course, if \m{X, Y \in \maV := \varrho_*(\Gamma(\oM;
  A)}), then \m{(X, Y) \in \CI(\oM)}. We now want to investigate some
properties of the metric \m{g}. For simplicity, we write \m{\Gamma(TM)
  = \Gamma(M; TM)} and \m{\varrho_*(\Gamma(A)) = \Gamma(\oM; A)}.
Let us consider the Levi-Civita connection
\begin{equation}\label{eq.def.LC}
  \nabla \sp{g} : \Gamma(TM) \, \to \, \Gamma(TM \otimes T^*M) 
\end{equation}
associated to the metric $g$.  Recall that an \m{A}--connection on a
vector bundle \m{E \to \oM} (see \cite{aln1} and the references
therein) is given by a differential operator $\nabla$ such that
\begin{equation}
  \nabla_{fX}(f_1 \xi) \, = \, f \big ( f_1\nabla_{X}(\xi) + X(f_1) \xi \big )
\end{equation}
for all \m{f, f_1 \in \CI(\oM)} and \m{\xi \in \Gamma(\oM; E)}.  The
following proposition from \cite{aln1} gives that the Levi-Civita
connection extends to an \dlp $A$-connection.\drp

\begin{proposition}\label{prop.LC}
Let $(\oM, A)$ be an open manifold with a Lie structure at infinity
and $g$ be a compatible metric $M$. The Levi-Civita connection
associated to the compatible metric $g$ extends to a linear
differential operator \m{\nabla = \nabla\sp{g} : \varrho_*(\Gamma(A))
  \to \Gamma(A \otimes A^*),} satisfying
\begin{enumerate}[(i)]
  \item \m{\nabla_X(fY) = X(f) Y + f \nabla_X(Y)},
  \item \m{X(Y,Z) = (\nabla_XY,Z) + (Y,\nabla_XZ)}, and
  \item \m{\nabla_X Y - \nabla_Y X = [X, Y]},
\end{enumerate}
for all \m{X, Y, Z \in \maV = \varrho_*(\Gamma(A))}.
\end{proposition}

\begin{proof} We recall the proof for the benefit of the reader.
Since the metric \m{g} actually comes from an metric on \m{A} by
restriction to \m{TM \subset A}, we see that
\begin{equation}\label{eq.def.nxy}
  \phi(Z) \, := \, ([X,Y],Z) - ([Y,Z],X) + ([Z,X],Y) + X (Y,Z) + Y
  (Z,X) - Z (X,Y) \,,
\end{equation}
defines a smooth function on \m{\oM} for any \m{Z \in \maV} and that
this smooth function depends linearly on \m{Z}. Hence there exists a
smooth section \m{V \in \maV} such that \m{\phi(Z) = (V, Z)} for all
\m{Z \in \maV}. We then define \m{\nabla_X Y := V}. By the definition
of \m{\nabla} and by the classical definition of the Levi-Civita
connection, \m{\nabla} extends the Levi-Civita connection. Since the
Levi-Civita connection satisfies the properties that we need to prove
(on \m{M}), by the density of \m{M} in \m{\oM}, we obtain that
\m{\nabla} satisfies the same properties.
\end{proof}

We continue with some remarks

\begin{remark}
An important consequence of the above proposition is that each of the
covariant derivatives \m{\nabla^k R} of the curvature \m{R} extends to
a tensor defined on the {\em whole of} \m{\oM}. If $\oM$ is compact
(that is, if \m{(\oM, A)} is a Lie manifold), it follows that the
curvature and all its covariant derivatives are {\em bounded}. It
turns out also that the radius of injectivity of \m{M} is positive
\cite{sobolev, Debord1}, and hence \m{M} has {\em bounded geometry}.
\end{remark}

We next discuss the divergence of a vector field, which is needed to
define adjoints.

\begin{remark}
Another important consequence of the existence of an extension of the
Levi-Civita connection to $\oM$ is the definition of the divergence of
a vector field. Indeed, let us fix a point \m{p \in \oM} and a local
orthonormal basis $X_1, \ldots, X_n$ of $A$ on some neighborhood of
\m{p} in $\oM$ (\m{n = \dim(\oM)}). We then write $\nabla_{X_i}X =
\sum_{j=1}\sp{n} c_{ij}(X) X_j$ and define
\begin{equation}\label{eq.def.div}
  \div(X) \, := \, -\sum_{j=1}\sp{n} c_{jj}(X) \,,
\end{equation} 
which is a smooth function on the given neighborhood of \m{p} that
does not depend on the choice of the local orthonormal basis \m{(X_i)}
used to define it. Consequently, this formula defines a global
function \m{\div(X) \in \CI(\oM)}.
\end{remark}

We now introduce differential operators on open manifolds with a Lie
structure at infinity. The desire to study these operators is the main
reason why we are interested in Lie manifolds.

\begin{definition} Let $(\oM, A)$ be an open manifold with a 
Lie structure at infinity and $\maV := \Gamma(\oM; A)$.  The algebra
\m{\Diff(\maV)} is the algebra of differential operators on \m{\oM}
generated by the operators of multiplication with functions in
\m{\CI(\oM)} and by the directional derivatives with respect to vector
fields \m{X \in \maV}.
\end{definition}

In \cite{ASkandalis1}, the definition of $\Diff(\maV)$ was extended to
the case of not necessarily projective $\CI(\oM)$-modules $\maV$.

Clearly, in our first example, Example \ref{ex.one}, the resulting
algebra of differential operators \m{\Diff(\maV) = \Diff(\maV_b)} for
$\oM$ a manifold with boundary is the algebra of totally
characteristic differential operators.  We shall see several other
examples in this paper.  The differential operators in \m{\Diff(\maV)}
can be regarded as acting either on functions on \m{\oM} or on
functions on \m{M := \oM \smallsetminus \pa \oM}. When it comes to
classes of measurable functions--say Sobolev spaces--this makes no
difference. However, the fact that \m{\Diff(\maV)} maps \m{\CI(\oM)}
to \m{\CIc(\oM)} is a non-trivial property that does not follow from
the $L\sp{2}$-mapping properties of \m{\Diff(\maV)} on \m{M}. We have
the following simple remark on the local structure of operators in
\m{\Diff(\maV)}.

\begin{remark}
Every \m{P \in \Diff(\maV)} of order at most \m{m} can be written as a
sum of differential monomials of the form \m{X_1\sp{\alpha_1}
  X_2\sp{\alpha_2} \ldots X_k\sp{\alpha_k}}, where \m{X_i \in \maV},
\m{k \le m}, and \m{\alpha} is a multi-index. If \m{Y_1, Y_2, \ldots,
  Y_n} are vector fields in \m{\maV} forming a local basis around \m{p
  \in \oM} (so \m{\dim(M) = n}), then every \m{P \in \Diff(\maV)} of
order at most \m{m} can be written in a neighborhood of \m{p} in
\m{\oM} uniquelly as
\begin{equation*}
  P \, = \, \sum_{|\alpha| \le m} a_{\alpha} Y_1\sp{\alpha_1}
  Y_2\sp{\alpha_2} \ldots Y_n\sp{\alpha_n} \,, \quad a_\alpha \in
  \CIc(\oM)\,.
\end{equation*}
This follows from the Poincar\arp{e}-Birkhoff-Witt theorem of
\cite{NWX}.
\end{remark}


The next remark states that the algebra $\Diff(\maV)$ is closed under
adjoints.

\begin{remark}\label{rem.L2}
We shall denote the inner product on \m{L\sp{2}(M; \vol_g)} by \m{(\ ,
  \ )_{L\sp{2}}}. Let \m{P \in \Diff(\maV)}. The formal adjoint
$\sh{P}$ of \m{P} is then defined by
\begin{equation}
  (P f_1, f_2)_{L\sp{2}} \, = \, (f_1, \sh{P}f_2)_{L\sp{2}}, \qquad
  f_1, f_2 \in \CIc(M).
\end{equation}
Let $X \in \maV := \varrho_*(\Gamma(A))$. Since \m{\div(X) \in
  \CI(\oM)} of Equation \eqref{eq.def.div} extends the classical
definition on \m{M}, we have that
\begin{equation*}
  \int_{M} \, X(f) \ d\vol_g \, = \, \int_{M}\, f \div(X) \ d\vol_g
  \,.
\end{equation*}
In particular, the formal adjoint of $X$ is
\begin{equation}
  \sh{X} \, = \, - X + \div(X) \in \Diff(\maV) \,,
\end{equation}
and hence \m{\Diff(\maV)} is closed under formal adjoints. 
\end{remark}

We can consider matrices of operators in $\Diff(\maV)$ and operators
acting on bundles.

\begin{remark}\label{rem.bundle}
We can extend the definition of \m{\Diff(\maV)} by considering the
space \m{\Diff(\maV; E, F)} of operators acting between smooth
sections of the vector bundles \m{E, F \to \oM}. This can be done
either by embedding the vector bundles \m{E} and \m{F} into trivial
bundles or by looking at a local basis. The formal adjoint of \m{P \in
  \Diff(\maV; E, F)} is then an operator \m{\sh{P} \in \Diff(\maV;
  F\sp{*}, E\sp{*})}. We shall write \m{\Diff(\maV; E) := \Diff(\maV;
  E, E)}.  Typically \m{E} and \m{F} will have hermitian metrics and
then we identify \m{E\sp*} with \m{E} and \m{F\sp*} with \m{F}. In
particular, if \m{E} is a Hermitian bundle, then \m{\Diff(\maV; E)} is
an algebra closed under formal adjoints.
\end{remark}

We are ready now to prove that all geometric operators on \m{M} that
are associated to a compatible metric \m{g} are generated by \m{\maV
  := \varrho_*(\Gamma(A))} \cite{aln1}. (Recall that a compatible
metric on \m{M} is a metric coming from a metric on the Lie algebroid
\m{A} of our Lie manifold \m{(\oM, A)} by restriction to \m{TM}.) In
particular, we have the following result \cite{aln1}.

\begin{proposition}\label{prop.geometric} 
We have that the de Rham differential \m{d} on \m{M} extends to a
differential operator \m{d \in \Diff(\maV; \Lambda^q A^*,
  \Lambda^{q+1} A^*)}. Similarly, the extension \m{\nabla} of the
Levi-Civita connection to an \m{A}-valued connection defines a
differential operator \m{\nabla \in \Diff(\maV; A, A \otimes A^*).}
\end{proposition}

\begin{proof} 
The proof of this theorem is to see that the classical formulas for
these geometric operators extend to \m{\oM}, provided that \m{TM} is
replaced by \m{A}. For instance, for the de Rham differential, let
\m{\omega \in \Gamma(\oM; \Lambda\sp{k} A\sp{*})} and \m{X_0,\ldots,
  X_k \in \maV}, and use the formula
\begin{multline*}
  (d\omega)(X_0,\ldots, X_k) \ = \ \sum_{j=0}^q \, (-1)^j
  X_j(\omega(X_0, \ldots, \hat{X}_j, \ldots, X_k)) \ + \\
  \sum_{0 \le i < j \le q} \, (-1)^{i+j} \omega([X_i,X_j], X_0,
  \ldots, \hat{X}_i, \ldots, \hat{X}_j, \ldots, X_k) \in \CI(\oM) \,.
\end{multline*}
By choosing \m{X_0,\ldots, X_k} among a local basis of \m{\maV :=
  \Gamma(\oM; A)} and using the fact that \m{\maV} is closed under the
Lie bracket, we obtain that \m{d \in \Diff(\maV; \Lambda^q A^*,
  \Lambda^{q+1} A^*)}, as claimed.

For the Levi-Civita connection, it suffices to show that \m{\nabla_X
  \in \Diff(\maV; A)} for all \m{X \in \maV} (we could even restrict
\m{X} to a local basis, but this is not really necessary). We will
show that, in a local basis of \m{A}, the differential operator
\m{\nabla_X} is given by an operator involving only derivatives in
\m{\maV}. To this end, we shall use the formula \eqref{eq.def.nxy}
defining \m{\nabla_X}, but choose \m{Y = fY_0} for \m{f \in \CI(\oM)}
and \m{Y_0} and \m{Z} in a local basis in some neighborhood \m{V} of
an arbitrary point \m{p \in \oM}. Then we see from formula
\eqref{eq.def.nxy}, using also the linearity of \m{\phi} in \m{Z},
that \m{(\nabla_X(Y), Z) = fz_0 + X(f)z_1}, with \m{z_0} and \m{z_1}
smooth functions on the given neighborhood \m{V}. Since the only
derivative in this formula is \m{X} and \m{X \in \maV}, this proves
the desired statement for \m{\nabla}.
\end{proof}

Let us consider a vector bundle \m{E \to \oM}. If \m{E} has a metric,
then an \m{A}--connection \m{\nabla \in \Diff(\maV; E, E \otimes
  A\sp{*})} is said to {\em preserve the metric} if
\begin{equation}
  (\nabla_X(\xi), \zeta)_E + (\xi, \nabla_X(\zeta))_E \, = \, X(\xi,
  \zeta)_E
\end{equation}
for all \m{X \in \maV} and \m{\xi, \zeta \in \Gamma(\oM; E)}. In
particular, it follows from Proposition \ref{prop.LC} that the
extension of the Levi-Civita connection to an \m{A}--connection on
\m{A} preserves the metric used to define it. We then have the
following theorem

\begin{theorem}\label{thm.geometric} We continue to 
consider the fixed metric on \m{A} and its associated compatible
metric \m{g} on \m{M}. Let \m{E \to \oM} be a hermitian bundle with a
metric preserving \m{A}-connection. Then \m{\Delta_E :=
  \nabla\sp*\nabla \in \Diff(\maV; E)}.  Similarly,
\begin{equation*}
	\Delta_g \, := \, d\sp*d \, \in \, \Diff(\maV).
\end{equation*}
\end{theorem}

\begin{proof} This follows from the fact that \m{\Diff(\maV)}
and its vector bundle analogues are closed under formal adjoints and
from Proposition \ref{prop.geometric}.
\end{proof}

Thus, in order to study geometric operators on a Lie manifold, it is
enough to study the properties of differential operators generated by
\m{\maV}. It should be noted, however, that a Riemannian manifold may
have different compactifications to a Lie manifold. An example is
\m{\RR\sp{n}}, which can either be compactified to \m{([-1, 1]\sp{n},
  \maV_b)} (a product of manifolds with cylindrical ends) or it can be
radially compactified to yield an asymptotically euclidean
manifold. (See Example \ref{ex.three}.)

Theorem \ref{thm.geometric} also gives the following.

\begin{remark}
Similarly, since \m{d \in \Diff(\maV; E, F)}, we get that the Hodge
operator \m{d + d\sp{*} \in \Diff(\maV; \Lambda^{*} A^*)}, the same is
true for the signature operator.  More generally, let \m{W \to \oM} be
a Clifford bundle with admissible connection in \m{\Diff(\maV ; W, W
  \otimes A\sp{*})}. Then the associated Dirac operators are also
generated by \m{\maV}.  If $W$ is the Clifford bundle associated to a
Spin${}\sp{c}$-structure on $A$, then the Levi-Civita connection on
$W$ is in \m{\Diff(\maV ; W, W \otimes A\sp{*})}.  All these
statements seem to be more difficult to prove directly in local
coordinates.  See \cite{aln1, LNGeometric} for proofs.
\end{remark}

\subsection{Anisotropic structures}
It is very important in applications to extend the previous frameworks
to include anisotropic structures \cite{BNZ3D2}. We introduce them now
for the purpose of later use.

\begin{definition} An {\em anisotropic structure}
on an open manifold \m{(\oM, A)} with a Lie structure at infinity is
an open manifold manifold with a Lie structure at infinity \m{(\oM,
  B)} (same $\oM$) together with a vector bundle map \m{A \to B} that
is the identity over \m{M} and makes \m{\Gamma(A) = \Gamma(\oM; A)} an
ideal (in Lie algebra sense) of \m{ \Gamma(B) = \Gamma(\oM; B)}.
\end{definition}

We shall denote \m{\maW := \varrho_*(\Gamma(B))}, so \m{\maV :=
  \varrho_*(\Gamma(A))} satisfies \m{[X, Y] \subset \maV} for all \m{X
  \in \maW} and \m{Y \in \maV}. Recall the groups \m{\exp(\maV)} and
\m{\exp(\maW)} introduced at the end of Subsection \ref{ssec.liealg}
(and recall that they are generated by compactly supported vector
fields), then we have the following.

\begin{remark}
 The group \m{\exp(\maW)} acts on \m{\oM} by Lipschitz
 diffeomorphisms, it acts on \m{A} by Lie algebroid morphism, it acts
 on \m{\maV := \varrho_*(\Gamma(A))} by Lie algebra morphisms, and it
 acts on on \m{\Diff(\maV)} by algebra morphisms. Moreover,
\begin{equation*}
  \exp(\maV) \subset \exp(\maW)
\end{equation*}
is a normal subgroup.
\end{remark}

\section{Analysis on Lie manifolds\label{sec4}}

Our main interest is in the analytic properties of the differential
operators in \m{\Diff(\maV)}. In this section, we introduce our
function spaces following \cite{sobolev} and discuss Fredholm
conditions. Throughout this section, \m{(\oM, A)} will denote an open
manifold with a Lie structure at infinity and Lie algebroid \m{A} with
anchor \m{\varrho : A \to T\oM}. Also, by \m{\maV :=
  \varrho_*(\Gamma(A))} we shall denote the structural Lie algebra of
vector fields on \m{\oM}.

\subsection{Function spaces}
We review in this subsection the needed definitions of function
spaces. Let \m{(\oM, A)} be our given open manifold with a Lie
structure at infinity and let \m{g} be a compatible metric on the
interior \m{M} of \m{\oM} (that is, coming from a metric on \m{A}
denoted with the same letter, see Definition
\ref{def.compatible}). Let \m{\nabla} be the Levi-Civita connection
acting on the tensor powers of the bundles \m{A} and \m{A\sp{*}}. We
then define, for \m{m \in \ZZ_+}, the Sobolev spaces as in
\cite{amann, sobolev, nadine, HebeyBook}:
\begin{equation}\label{eq.def.Sobolev}
	H^m(M) \ = \ \{ u : M \to \CC, \ \nabla\sp{k} u \in L^2(M;
        A\sp{*\otimes k}), \ 0 \le k \le m \, \} \,.
\end{equation}

\begin{remark}
In general, the Sobolev spaces \m{H\sp{m}(M)} will depend on the
choice of the metric \m{g}, but if \m{\oM} is compact (that is, if
$(\oM, A)$ is a Lie manifold), then they are independent of the choice
of the metric, as we shall see below. It is interesting to notice that
if denote by \m{d\vol_g} the volume form (1-density) associated to
\m{g}. If \m{h} is another such compatible metric, then
\m{d\vol_h/d\vol_g} and \m{d\vol_g/d\vol_h} extend to smooth, bounded
functions on \m{\oM}. Hence the space \m{L^2(M) := L\sp{2}(M; \vol_g)}
is independent of the choice of the compatible metric \m{g}.
\end{remark}

The spaces \m{H\sp{m}(M)} behave well with respect to anisotropic
structures.

\begin{proposition}
Let \m{(\oM, A)} be an open manifold with a Lie structure at infinity
and with an anisotropic structure \m{(\oM, B)}, such that \m{\maW :=
  \Gamma(B) \supset \Gamma(A)}. Then \m{\exp(\maW)} acts by bounded
operators on \m{H\sp{m}(M)}.
\end{proposition}

\begin{proof}
This follows from the fact that \m{\exp(\maW)} is generated by vector
fields with compact support in \m{\oM}.
\end{proof}

We now consider some alternative definitions of these Sobolev spaces
in particular cases. We first consider the case of complete manifolds
\cite{amann, sobolev, nadine, HebeyBook, Hebey1}.

\begin{remark} 
Let us assume that \m{(\oM, A)} and the compatible metric \m{g} are
such that \m{M} is complete and let \m{\Delta_g} be (positive)
Laplacian associated to the metric \m{g}. Then \m{H^s(M)} coincides
with the domain of \m{(1 + \Delta_g)^{s/2}} (we use the geometer\rp s 
Laplacian, which is positive).
\end{remark}

In the bounded geometry case we can consider partitions of unity.

\begin{remark}
Let us assume that \m{(\oM, A)} and the compatible metric \m{g} are
such that \m{M := \oM \smallsetminus \pa \oM} is of bounded geometry.
Then the definition of the Sobolev spaces on \m{M} can be given using
a choice of partition of unity with bounded derivatives as in
\cite{sobolev}, for example, to patch the locally defined classical
Sobolev spaces. See also \cite{amann, sobolev, Disconzi, nadine,
  kordyukovLp1, kordyukovLp2, ShubinAsterisque}.
\end{remark}

Finally, if \m{\oM} is an open subset of a Lie manifold, we have yet
the following definition.

\begin{remark}
Let \m{U \subset \oM} be an open subset of a Lie manifold $(\oM, A)$
(so $\oM$ is compact) and let, as usual, $\maV$ denote the structural
Lie algebra of vector fields $\Gamma(\oM, A)$, then \cite{aln1}
\begin{equation}\label{eq.def.Sobolev2}
	H^m(U) \ := \ \{ u : U \to \CC, \ X_1 X_2 \ldots X_k u \in
        L^2(U), \ k \le m, \ X_j \in \maV\} \,,
\end{equation}
so the Sobolev spaces \m{H\sp{m}(U)} will be independent of the chosen
compatible metric on the Lie manifold.
\end{remark}

Each of theses definitions of the Sobolev spaces has its own
advantages and disadvantages.  For instance, the definition
\eqref{eq.def.Sobolev2} has the advantage that it immediately gives
the boundedness of operators \m{P \in \Diff(\maV)}, see Lemma
\ref{lemma.bg}. Let \m{H^{-s}(M) := (H^s(M))^*} and extend the
definition of Sobolev space to \m{s} non-integer by interpolation.

\begin{lemma}\label{lemma.bg}
Let us assume that \m{A} is endowed with a metric such that the
resulting metric \m{g} on \m{TM \subset A} is of bounded geometry. Let
\m{P \in \Diff(\maV)} of order \m{\ord(P) \le m} and with coefficients
that are compactly supported in $\oM$. Then the map \m{P : H^{s}(M)
  \to H^{s-m}(M)} is bounded for all \m{s \in \RR}.
\end{lemma}

Note that if $\oM$ is compact (i.e. $(\oM, A)$ is a Lie manifold),
then all \m{P \in \Diff(\maV)} have compactly supported coefficients.
Let us denote by \m{(E)_r} the set of vectors of length \m{<r}, where
\m{E} is a real or complex vector bundle endowed with a metric.

\begin{proof} 
Let \m{K \subset \oM} be a compact subset such that the coefficients
of \m{P} are zero outside \m{K}. Let us choose a compact neighborhood
\m{L} of \m{K} in \m{\oM} and let \m{r_0} be the distance from \m{K}
to the complement of \m{L} in a metric \m{h} on \m{\oM} such that \m{h
  \le g}, which exists by Lemma \ref{lemma.distance}.  Then \m{r_0 >
  0}, because \m{K} is compact.  Moreover, the distance from \m{K} to
the complement of \m{L} in the metric \m{g} is \m{\ge r_0} since \m{h
  \le g}. Let us fix $r$ less than the injectivity radius of \m{M} and
with \m{r_0 > r > 0}. For every \m{p \in M}, we then consider the
exponential map \m{\exp: T_p M \to M}, which is a diffeomorphism from
(the open ball of radius \m{r}) \m{(T_pM)_r} onto its image. Thus
\m{P} gives rise to a differential operator \m{P_p} on each of the
open balls \m{(T_pM)_r}. Using the results of \cite{ShubinAsterisque},
it suffices to show that the coefficients of \m{P} in any of these
balls of radius \m{r} are uniformly bounded. Indeed, this is a
consequence of the following lemma, where the support of the resulting
map is contained in \m{L}.
\end{proof}

The following lemma (see \cite{aln2}) underscores the additional
properties that the Lie manifolds enjoy among the class of all
manifolds with bounded geometry.

\begin{lemma}\label{lemma.compact} Let us use the notation of 
the proof of the previous lemma and denote for any \m{p \in M} by
\m{P_p} the differential operator on \m{(T_pM)_r} induced by the
exponential map. Then the map \m{M \ni p \to P_p} extends to a
compactly supported smooth function defined on \m{\oM} such that
\m{P_p} is a differential operator on \m{(A_p)_r}.
\end{lemma}

We define the {\em anisotropic} Sobolev spaces in a similar way.

\begin{remark}
Let \m{(\oM, A)} be a Lie manifold with an anisotropic structure
\m{(\oM, B)}, and let \m{\maW := \Gamma(B) \supset \Gamma(A)}.  Then
we define
\begin{multline}\label{eq.def.Sobolev3}
	H^{m+q}_{\maW}(M) \ := \ \{ u : M \to \CC, \ X_1 X_2 \ldots
        X_k Y_1, Y_2 \ldots Y_l u \in L^2(M), \\
          \text{for all }\ X_j \in \maW \, , \ 0 \le j \le k \le m,
          \ \mbox{ and for all }\ Y_i \in \maV \,, \ 0 \le i \le l \le
          q \, \} \,.
\end{multline}
The spaces \m{H_{\maW}\sp{m+q}(M)} are again independent of the chosen
compatible metric on the Lie manifold.
\end{remark}

\subsection{Pseudodifferential operators on Lie manifolds}

Let us begin by recalling the definition of a tame submanifold with
corners from \cite{aln1} (we changed slightly the terminology).

\begin{definition}\label{def.sub.corners}
Let $\oM$ be a manifold with corners and $L \subset \oM$ be a
submanifold. We shall say that $L$ is a {\em tame submanifold with
  corners} of $\oM$ if $L$ is a manifold with corners (in its own)
that intersects transversely all faces of $\oM$ and such that each
open face $F_0$ of $L$ is the open component of a set of the form $F
\cap L$, where $F$ is an open face of $\oM$ (of the same codimension
as $F_0$).
\end{definition}

The closed faces of a manifold with corners $\oM$ are thus {\em not}
tame submanifolds with corners of $\oM$ even if they happen to be
manifolds with corners. Also, the diagonal of the $n$-dimensional cube
$[-1, 1]^n$ is not a tame submanifold with corners. However, $\{0\}
\times [-1, 1]^{n-1}$ is a tame submanifold with corners of $[-1,
  1]^n$. In fact, all tame submanifolds with corners \m{L \subset \oM}
have a tubular neighborhood \cite{sobolev, aln2}. This tubular
neighborhood allows us{ then to define the space \m{I\sp{m}(\oM, L)}
of {\em classical conormal distributions} as in \cite{hormander3} or
as in \cite{NWX} for manifolds with corners as follows.  Let $V \to X$
be a real vector bundle. A distribution on $V$ is (classically) {\em
  conormal} to $X$ if its fiberwise Fourier transform is a {\em
  classical} symbol on $V\sp{*} \to M$. We shall denote the set of
these distributions corresponding to a symbol of order $\le m$ by
$I\sp{m}(V, M)$. We shall denote by $I_c\sp{m}(V, M) \subset
I\sp{m}(V, M)$ the subset of conormal distributions with compact
support. We extend the definition of $I_c\sp{m}(V, M)$ to the case of
a tame submanifold $M \subset V$ by localization.

Let us fix a compatible metric on \m{M}, the interior of our open
manifold with a Lie structure at infinity \m{(\oM, A)}. Also, let us
fix \m{r>0} less than the injectivity radius of \m{M}. As in the proof
of Lemma \ref{lemma.bg}, the exponential map then defines a
diffeomorphism from the set \m{(TM)_r} of vectors of length \m{< r} to
\m{(M \times M)_r}, which is an open neighborhood of the diagonal in
\m{M \times M}.  This allows us to define a natural bijection
\begin{equation}
    \Phi : I_c\sp{m} ((TM)_r, M) \to I_c\sp{m} ((M \times M)_r, M) \,,
\end{equation}
Similarly, we obtain by restriction an inclusion
\begin{equation}
    I_c\sp{m} ((A)_r, \oM) \to I\sp{m} ((TM)_r, M) \,.
\end{equation}
Recall the group of diffeomorphisms \m{\exp(\maV)} defined at the end
of Subsection \ref{ssec.liealg}. Then we define as in \cite{aln2}
\begin{equation}
  \Psi_{\maV}\sp{m}(M) \ := \ \Phi(I_c\sp{m} ((A)_r, \oM)) \, + \,
  \Phi(I_c\sp{-\infty} ((A)_r, \oM))\exp(\maV) \,.
\end{equation}

Then $\Psi_{\maV}\sp{m}(M)$ is independent of the parameter \m{r > 0}
or the compatible metric $g$ used to define it and we have the
following result \cite{aln2}.

\begin{theorem} Let $(\oM, A)$ be a Lie manifold, $M := 
\oM \smallsetminus \pa \oM$ and $\maV := \varrho_*(\Gamma(A))$.  We
have \m{\Psi_{\maV}\sp{m}(M)\Psi_{\maV}\sp{m\rp}(M) \subset
  \Psi_{\maV}\sp{m+m\rp}(M)}. The subspace \m{\Psi_{\maV}\sp{m}(M)} is
closed under adjoints and the principal symbol \m{\sigma_m :
  \Psi_{\maV}\sp{m}(M) \to S_{cl}\sp{m}(A\sp*)/S_{cl}\sp{m-1}(A\sp*)}
is surjective with kernel \m{\Psi_{\maV}\sp{m-1}(M)}.  Moreover, any
\m{P \in \Psi_{\maV}\sp{m}(M)} defines a bounded operator
\m{H\sp{s}(M) \to H\sp{s-m}(M)}. If an anisotropic structure with
structural vector fields $\maW$ is given, then the group
\m{\exp(\maW)} acts by degree preserving automorphisms on
\m{\Psi_{\maV}\sp{m}(M)}.
\end{theorem}

The group $\exp(\maW)$ is seen to act on \m{\Psi_{\maV}\sp{m}(M)}
since it acts by Lipschitz diffeomorphisms of $\oM$.  The proof of the
above theorem is too long to include here.  See \cite{aln2} for
details. Let us just say that it is obtained by realizing
\m{\Psi_{\maV}\sp{*}(M)} as the image of a groupoid pseudodifferential
operator algebra \cite{aln2, Monthubert, MonthubertPierrot, NWX} for
{\em any} Lie groupoid integrating the Lie algebroid \m{A} defining
the Lie manifold \m{(\oM, A)} \cite{Debord1, Debord2, NistorJapan}.

The algebra \m{\Psi_{\maV}\sp{*}(M)} has the property that its subset
of differential operators coincides with \m{\Diff(\maV)}. It also has
the good symbolic properties that answer to a question of Melrose
\cite{aln2, MelroseICM}.

\subsection{Comparison algebras}
\label{ssec.comparison}
We continue to denote by \m{(\oM, A)} an open manifold with a Lie
structure at infinity and by \m{\maV:=\varrho_*(\Gamma(A))}.  For
simplicity, we shall assume that \m{\oM} is connected.  We now recall
from \cite{MN3} the {\em comparison} \m{C\sp{*}}-algebra \m{\mfk{A}(U,
  \maV)} associated to an open subset $U \subset \oM$. Its definition
extends to open manifolds with a Lie structure at infinity by Lemma
\ref{lemma.bg} that justifies the following definition.

\begin{definition}\label{def.comparison}
Let us assume that \m{A} has a metric \m{g} such that the induced
metric on \m{M} is of bounded geometry and let \m{U \subset \oM} be an
open subset. Then \m{\mfk{A}(U; \maV)} is the norm closed subalgebra
of the algebra \m{\maB(L\sp{2}(M; \vol_g))} of bounded operators on
\m{L\sp{2}(M; \vol_g)} generated by all the operators of the form
\m{\phi_1 P (1 + \Delta_g)\sp{-k} \phi_2}, where \m{\phi_i \in
  \CIc(U)}, \m{P \in \Diff(\maV)} is a differential operator of order
\m{\le 2k}, and $\Delta_g$ is the Laplacian on $M$ (not on $U$!)
associated to the metric $g$.
\end{definition}

In case an anisotropic structure is given, the group \m{\exp(\maW)}
acts by automorphisms on the comparison algebra \m{\mfk{A}(\oM;
  \maV)}.  We shall need the following lemma that follows right away
from the results in \cite{aln1} and \cite{LNGeometric}.  By a
pseudodifferential operator on a manifold, we shall mean one that is
obtained by the usual quantization formula in any coordinate system.

\begin{lemma}\label{lemma.princSymb}
Let us use the notation and the assumptions of Definition
\ref{def.comparison} and let \m{T := \phi_1 P (1 + \Delta_g)\sp{-k}
  \phi_2}. Then \m{T} is contained in the norm closure of
\m{\Psi_{\maV}\sp{0}(M)} and is a pseudodifferential operator of order
\m{\le 0} with principal symbol
\begin{equation*}
  \sigma_0(T) = \phi_1 \sigma_{2k}(P) (1 + |\xi|\sp{2})\sp{-k} \phi_2
  \,.
\end{equation*}
Moreover, the principal symbol depends continuously on \m{T}, and
hence extends to a continuous, surjective morphism \m{\sigma_0 :
  \mfk{A}(U; \maV) \to \maC_0(S\sp{*}A\vert_{U})}.
\end{lemma}

As in \cite{MN3}, we obtain the following result.

\begin{theorem} \label{theorem.3}
Let $(\oM, A)$ be a connected open manifold with a Lie structure at
infinity with a compatible metric of bounded geometry.  Then \m{\fa(M;
  \maV)} contains the algebra \m{\maK(L\sp{2}(M))} of all compact
operators on \m{L\sp{2}(M)} and is contained in the norm closure of
\m{\Psi_{\maV}\sp{0}(M)}.
\end{theorem}

\begin{proof} We recall the proof for the benefit of
the reader. The inclusion of \m{\fa(\oM; \maV)} in the norm closure of
\m{\Psi_{\maV}\sp{0}(M)} follows from Lemma \ref{lemma.princSymb}.

Let \m{\phi_1, \phi_2 \in \CIc(M)} and \m{P \in \Diff(\maV)} be a
differential operator of order at most $2k-1$, then the composition
operator \m{\phi_1 P (1 + \Delta)\sp{-k} \phi_2} is a non-zero compact
operator and belongs to \m{\fa(\oM; \maV)}, by the definition.  The
role of the cut-off functions is to decrease the support of the
distribution kernel of $P (1 + \Delta)\sp{-k}$.  Since $\phi_1 P (1 +
\Delta) \sp{-k} \phi_2$ is compact, we thus obtain that \m{\fa(\oM;
  \maV)} contains non-zero compact operators. Let \m{\xi_1, \xi_2 \in
  L\sp{2}(M)} be nonzero.  Then we can find \m{\phi_1, \phi_2}, and
\m{P} as above such that $T := \phi_1 P (1 + \Delta)\sp{-k} \phi_2$
satisfies \m{(T \xi_1, \xi_2) \neq 0}. Hence \m{\fa(\oM; \maV)} has no
non-trivial invariant subspace.  Hence \m{\fa(\oM; \maV)} contains all
compact operators because any proper subalgebra of the algebra of
compact operators has an invariant subspace.
\end{proof}

\subsection{Fredholm conditions}
Theorem \ref{theorem.3} allows us, in principle, to study the Fredholm
property of operators in \m{\fa(\oM; \maV)}.  Let us denote by \m{\maK
  = \maK(L\sp{2}(M)} the ideal of compact operators in \m{\fa(\oM;
  \maV)}. Recall then Atkinson\rp s classical result \cite{Douglas}
that states that \m{T \in \fa(\oM; \maV)} is Fredholm if, and only if,
its image \m{T + \maK} in \m{\fa(\oM; \maV)/\maK} is invertible.

Usually it is difficult to check directly that \m{T + \maK} is
invertible in \m{\fa(\oM; \maV)/\maK}, and, instead, one checks the
invertibility of operators of the form \m{\pi(T)}, where \m{\pi}
ranges through a suitable family of representations of \m{\fa(\oM;
  \maV)/\maK}. Exactly what are the needed properties of the family of
representations of \m{\fa(\oM; \maV)/\maK} was studied in
\cite{NistorPrudhon, Roch}. Let us recall the main conclusions of that
paper. Let us consider a family of representations \m{\maF} and \m{\pi
  \in \maF}. It is not enough for the family \m{\maF} to be faithful
so that the invertibility of all $\pi(T)$ implies the Fredholmness of
$T$. For this implication to be true, the necessary condition is that
the family \m{\maF} be {\em invertibility sufficient}
\cite{NistorPrudhon, Roch}.  An equivalent condition (in the separable
case) is that the family $\maF$ be {\em exhausting}, in the sense that
every irreducible representation of \m{\fa(\oM; \maV)/\maK} is weakly
contained in one of the representations \m{\pi \in \maF}.

This approach was used (more or less explicitly) in
\cite{DamakGeorgescu, DLR, GeorgescuIftimovici, GeorgescuNistor2,
  GeorgescuNistor1, LMN1, LNGeometric, MantoiuReine, bm-these, Rab3,
  Roch, RochBookNGT, SchroheFrechet}, and in many other
papers. However, in order for this approach to be effective, we need
to have a good understanding of the representation theory of the
quotient \m{\fa(\oM; \maV)/\maK}. This seems to be difficult in
general, at least without using groupoids. Thus we shall replace the
comparison algebra \m{\fa(\oM; \maV)} with the norm closure of the
algebra \m{\Psi_{\maV}\sp{0}(M)}. The algebra \m{\Psi_{\maV}\sp{0}(M)}
is defined in the next section.

There are many general results that yield Fredholm conditions for
operators.  We formulate now one such result that is sufficient in
most applications.  We shall make some the following {\bf assumptions
  on the Lie manifold} \m{(\oM, A)}
\begin{enumerate}[(a)]
\item We assume that there exists a filtration 
\begin{equation}\label{eq.filtration}  
  \emptyset =: U_{-1}
  \subset U_0 \subset \ldots \subset U_k\subset U_{k+1}
  \subset \ldots \subset U_N := \oM
\end{equation}
  of \m{\oM} with open sets such that each $S_k := U_{k}\smallsetminus
  U_{k-1}$ is a manifold (possibly with corners) and that there exists
  submersions \m{p_k : S_k \to B_k} of manifolds (possibly with
  corners) whose fibers are orbits of \m{\exp(\maV)}.
 
\item We assume that, for each $k = 0, \ldots, N$, there exists a Lie
  algebroid \m{A_k \to B_k} with zero anchor map such that
  \m{A\vert_{S_k} \simeq p_k\pullback(A_k)}, the pull-back of \m{A_k}
  by the submersion \m{p_k} \cite{HigginsMackenzie2}. In particular,
  we have the isomorphism of vector bundles
\begin{equation*}
  A\vert_{S_k} \, \simeq \, \ker(p_k)_* \oplus p_k\sp*(A_k) \,.
\end{equation*} 
  
\item Let us denote by \m{(Z_\alpha)_{\alpha \in J}} the family of
  {\em orbits} \m{Z_\alpha = \exp(\maV)p} of \m{\maV} and by
  \m{G_\alpha} the simply-connected Lie group that integrates the Lie
  algebra \m{(A_k)_{p_k(p)} \simeq \ker(\varrho_p)}, for any \m{p \in
    Z_\alpha \subset S_k}. Also, let us denote by \m{\maG} the
  disjoint union
\begin{equation}\label{eq.def.maG}
 \maG \ede  \,  \cup_{\alpha \in J} Z_\alpha
    \times Z_{\alpha} \times G_{\alpha} 
\end{equation} 
   with the induced groupoid structure. We assume that the groupoid
   exponential map makes \m{\maG} a Hausdorff Lie groupoid
   \cite{NistorJapan}. In particular, \m{\maG} is a manifold (possibly
   with corners).
\end{enumerate} 

Under the above assumptions, the results in \cite{KSkandalis, LMN1,
  LNGeometric, NistorPrudhon, renaultBook, renault91} give the
following theorem.

\begin{theorem} \label{thm.fc.Lie}
Let $I$ be an index parametrizing the set of obits of $\maV$ on $\pa
\oM$. (So $J = I \cup \{0\}$.)  We can associate to each \m{P \in
  \Diff(\maV; E, F)} a family $(P_\alpha)$, of \m{G_\alpha}-invariant
operators \m{P_\alpha} on \m{Z_\alpha \times G_\alpha}. If all the
groups $G_{\alpha}$, $\alpha \in I$ are amenable, then the following
Fredholm condition holds.
\begin{multline*}
  P : H\sp{s}(M) \to H\sp{s-m}(M) \mbox{ is Fredholm
  }\ \ \Leftrightarrow \ \ \ \m{P} \mbox{ is elliptic} \\ \mbox{ and
    all } \ P_\alpha : H\sp{s}(Z_\alpha \times G_\alpha ) \to
  H\sp{s-m}(Z_\alpha \times G_\alpha ), \, \alpha \in I,\ \mbox{ are
    invertible}\,.
\end{multline*}
\end{theorem}

\begin{proof}(Sketch)
The exact sequence of (full) $C\sp{\ast}$-algebras associated to an
open subset of the set of units of a groupoid \cite{renaultBook} tells
us that $\Prim(C\sp{\ast}(\maG))$ is the disjoint union of the sets
$\Prim(C\sp{\ast}(\maG_{S_k}))$. We have that $C\sp{\ast}(\maG)\simeq
C_r\sp{\ast}(\maG)$, because $C\sp{\ast}(\maG_{S_k})\simeq
C_r\sp{\ast}(\maG_{S_k})$ for each $k$, by the amenability of the
groups $G_\alpha$.  This shows that the set $\{\Ind(\lambda_\alpha)\}$
of representations of $C\sp{\ast}(\maG)$ induced from the regular
representations $\lambda_\alpha$ of $G_\alpha$ is an exhausting set of
representations of $C\sp{\ast}(\maG)\simeq C_r\sp{\ast}(\maG)$. Each
$\pi_{\alpha} := \Ind(\lambda_\alpha)$ is the regular representation
of $C_r\sp{\ast}(\maG)$ associated to (any unit in) the orbit
$Z_{\alpha}$, with $\pi_0$ being the vector representation on
$L\sp{2}(M)$.
 
The assumptions imply that $S_0 = U_0 = M$ and that $G_0 = \{e\}$.
Hence $\maG_{S_0}$ is the pair groupoid $S_0 \times S_0$ and
$C\sp{\ast}(\maG_{M}) \simeq \maK$, the algebra of compact operators
on $L\sp{2}(M)$. We obtain that $C\sp{\ast}(\maG)/\maK \simeq
C\sp{\ast}(\maG_{\pa \oM}) \simeq C_r\sp{\ast}(\maG_{\pa \oM})$ and
the regular representations of $C\sp{\ast}(\maG_{\pa \oM})$ form an
invertibility sufficient set of representations of
$C\sp{\ast}(\maG_{\pa \oM})$.  This gives that $a \in
C\sp{\ast}(\maG)\simeq C_r\sp{\ast}(\maG)$ is Fredholm if, and only
if, $\pi_\alpha(a)$ is invertible for all $\alpha$ corresponding to
orbits in $\pa \oM$.
 
We shall apply these observations to the algebra
$\Psi\sp{\infty}(\maG)$ of pseudodifferential operators on groupoids,
which is recalled in the next subsection.  The fact that $\maG$ is
Hausdorff implies that the vector representation of $C\sp{\ast}(\maG)$
(associated to the orbit $M \subset \oM$) is injective, by
\cite{KSkandalis}. We shall use then the vector representation to
identify $\Psi\sp{\infty}(\maG)$ and $C_r\sp{\ast}(\maG)$ with their
images under the vector representation. In, particular, $P$ is given
by a family of operators $(P_x)$, $x \in M$, with operators
corresponding to units in the same orbit unitarily equivalent. We then
let $P_\alpha := P_x$, for some $x$ in the orbit corresponding to
$\alpha$.
 
Let $a := (1 + \Delta)\sp{(s-m)/2} P (1 + \Delta)\sp{-s/2} \in
\overline{\Psi\sp{0}}(\maG)$ \cite{LMN1, LNGeometric}.  We then have
that $P$ is Fredholm if, and only if, $a$ is Fredholm on $L\sp{2}(M)$,
which, in turn, is true, if, and only if, $\pi_\alpha(a)$ is
invertible for all $Z_\alpha \subset \pa\oM$. Since
\begin{equation*}
 \pi_x(a) \ede (1 + \pi_x(\Delta))\sp{(s-m)/2} P_x (1 + \pi_x(\Delta))\sp{-s/2},
\end{equation*}
acting on $Z_\alpha \times G_\alpha$, the result follows from the fact
that the set of arrows of $\maG$ with domain $x \in Z_\alpha$ is
$Z_\alpha \times G_\alpha$.
\end{proof}

This theorem is closely related to the representations of Lie
groupoids, see \cite{buneciSurvey, Orloff07, OrloffHuef12,
  EchterhoffMemoirs, Echterhoff96, IonescuWilliamsEHC, Renault87,
  ErpWilliams}. For our result, we need Hausdorff groupoids, see
however also \cite{KSkandalis, TuNonH} for some results on non-Hausdorff
groupoids. More general Fredholm conditions can be
obtained along the same lines,
but the result mentioned here, although having a rather long list of
assumptions, is easy to prove and to use. More references to earlier
results will be given in the next section when discussing examples.

\begin{remarks}
We continue with a few remarks. 

\begin{enumerate}
\item If \m{\oM} is compact and smooth (so without corners), then \m{I
  = \emptyset}, and we recover Theorem \ref{thm.fc.as}. As we will
  explain below, we also recover Theorem \ref{thm.fc.aps2}.  Each
  operator \m{P_\alpha} is ``of the same kind'' as \m{P} (Laplace,
  Dirac, ... ) and can be recovered by ``freezing the coefficients''
  at the orbit \m{Z_\alpha}. The theorem allows us to reduce some
  questions on \m{M} to questions on \m{P_\alpha} and
  \m{G_\alpha}. Because of the \m{G_\alpha}-invariance of our
  operators, we can use results on harmonic analysis on \m{G_\alpha}
  to obtain an inductive procedure to study geometric operators on
  \m{M} \cite{Mehdi, Pasquale}.

\item Each open face $F_0$ of \m{\oM} is invariant for \m{\exp(\maV)},
  and hence, if an orbit \m{Z_\alpha} intersects \m{F_0}, then it is
  completely contained in \m{F_0}. In particular, the set of orbits
  \m{I} identifies with the disjoint union of the sets \m{B_k} for
  \m{k=1, 2, \ldots, N}.
  
\item The sections of \m{\ker(p_k)_*} act by derivation on the
  sections of \m{p_k\sp*(A_,)}. Also, it follows that for any \m{p \in
    F_0}, the isotropy Lie algebra \m{\ker(\varrho_p)} is canonically
  isomorphic to the Lie algebra \m{(A_k)_{p_k(p)}}, see Definition
  \ref{def.isotropy}.

\item We note that our assumptions on \m{(\oM, A)} imply that the
  groupoid \m{\maG} considered in our assumptions must coincide with
  the one introduced by Claire Debord \cite{Debord1, Debord2}.
\end{enumerate}
\end{remarks}


\subsection{Pseudodifferential operators on groupoids}\label{ssec.psdo}
Let us briefly recall, for the benefit of the reader, the definition
of pseudodifferential operators on a Lie groupoid \m{\maG}.  Let $d :
\maG \to \oM$ be the domain map and $\maG_x = d^{-1}(x)$. Then
$\Psi^{m}(\maG)$, $m \in \RR$, consists of smooth families \m{(P_x)_{x
    \in \oM}} of classical, order $m$ pseudodifferential operators
$(P_x \in \Psi^m(\maG_x))$ that are right invariant with respect to
multiplication by elements of $\maG$ and are ``uniformly supported.''
To define what uniformly supported means, let us observe that the
right invariance of the operators $P_x$ implies that their
distribution kernels $K_{P_x}$ descend to a distribution $k_P \in
I^m(\maG, \oM)$ \cite{bm-these, NWX}. The family $P = (P_x)$ is called
{\em uniformly supported} if, by definition, $k_P$ has compact support
in $\maG$.

Groupoids simplify the study of pseudodifferential operators on
singular and non-compact spaces. For instance, one obtains a
straightforward definition of the \dlp generalized indicial
operators\drp\ as restrictions to invariant subsets \cite{LMN1}. More
precisely, let \m{N \subset \oM} be an invariant subset for \m{\maG},
that is, \m{d\sp{-1}(N) = r\sp{-1}(N)}, and let \m{\maG_N :=
  d\sp{-1}(N)}. Let us now assume that \m{P \in \Psi\sp{m}(\maG)} is
given by the family \m{(P_x)_{x \in \oM}}, then the \m{N}--indicial
family \m{\maI_N(P) := (P_x)_{x \in N}} is defined simply as the
restriction of \m{P} to \m{N} and is in \m{\Psi\sp{m}(\maG_N)}.  See
\cite{DebordSkandalis} for an extension of these results in relation
to the adiabatic groupoid. See also \cite{Aastrup, Karsten, KarstenCR,
  MeloNestSchrohe} for results on the Boutet-de-Montvel calculus in
the framework of groupoids.

In order to be able to use the machinery of Lie groupoids in analysis,
one has to sometimes first {\em integrate} the Lie algebroid that
naturally arises in the analysis problem at hand. That is, given a Lie
algebroid $A \to \oM$, one wants to find a Lie groupoid $\maG$ such
that $A(\maG) \simeq A$. Such a groupoid $A$ does not always exist,
and when it exists, it is not unique. Moreover, the choice of the
groupoid $\maG$ depends on the analysis problem one is interested to
solve.  For example, the Lie algebroid $TM \to M$ for a smooth,
compact manifold $M$ has the pair groupoid $M \times M$ as a {\em
  minimal} integrating groupoid and has $\maP(M)$, the path groupoid
of $M$ as the {\em maximal} integrating groupoid \cite{Debord1,
  Debord2, NistorJapan}. The first groupoid leads to the usual
analysis on compact, smooth manifolds (the AS-framework), whereas the
second one leads to the analysis of invariant operators on $\widetilde
M$, the universal covering space of $M$ (with group of deck
transformation $\pi_1(M)$). Both these examples are examples of
$d$-connected integrating groupoids (i.e. the fibers of the domain map
$d$ are connected). There are analysis problems, however, when one is
lead to non-$d$-connected groupoids \cite{CarvalhoYu}.

There are many works dealing with pseudodifferential operators on
groupoids, on singular spaces, or with the related $C\sp{*}$-algebras,
see for example \cite{Alldridge, Beltitza1, LindnerMem, Coriasco, DLR,
  Grieser, LeschPflaum, Mantoiu3, Mantoiu1, NazaSavinSternin, Perez,
  RochBookLimit, Roch, RochBookNGT, BKSo2, BKSo1, Vassout}.

\section{Examples and applications\label{sec5}}

We now discuss some applications. They are included just to give an
idea of the many possible applications of Lie manifolds, so we will be
short, but we refer to the existing literature for more details.  We
begin with some examples.

\subsection{Examples of Lie manifolds and Fredholm conditions}
We now include examples of Lie manifolds and show how to use Theorem
\ref{thm.fc.Lie}.  The following examples cover many of the examples
appearing in practice.

\begin{example}\label{ex.one.bis}
We now review our first, basic example, Example \ref{ex.one}, in view
of the new results. Recall that \m{\maV= \maV_b :=} the space of
vector fields on \m{\oM} that are tangent to \m{\pa \oM}.  Near the
boundary, a local basis is given by Equation
\eqref{eq.local.basis.one} of Example \ref{ex.one}, and hence
\m{\Diff(\maV_b)} is the algebra of totally characteristic
differential operators.  If \m{\oM} has a smooth boundary and we
denote by \m{r} the distance to the boundary (in some everywhere
smooth metric--including the boundary), then a typical compatible
metric on \m{M} is given near the boundary by \m{(r\sp{-1} dr)\sp{2} +
  h}, where \m{h} is a metric smooth up to the boundary. Hence the
geometry is that of a manifold with cylindrical ends.

We have that the orbits \m{Z_\alpha} are the open faces of \m{\oM},
except \m{M} itself. The groups are \m{G_\alpha \simeq \RR\sp{k}},
where \m{k} is the codimension of the corresponding face (so all are
commutative Lie groups).  In the case of a smooth boundary, the
\m{Z_\alpha}\rp s are the connected components of the boundary,
\m{G_\alpha = \RR}, and \m{P_\alpha} is the restriction of \m{\maI(P)}
to a translation invariant operator on \m{Z_{\alpha} \times \RR}.  See
also \cite{Lesch, MelrosePiazza, MelroseMendoza, Monthubert,
  MonthubertPierrot, SchSch1} for just a sample of the many papers on
this particular class of manifolds.
\end{example}

In the following examples, \m{\oM} will be a compact manifold with
smooth boundary \m{\pa \oM}. The following example is that of an
asymptotically hyperbolic space and has the feature that it leads to
non-commutative groups \m{G_\alpha}.

\begin{example}\label{ex.two}
Let \m{\oM} be a compact manifold with smooth boundary \m{\pa \oM} and
defining function \m{r}. We proceed as in Example \eqref{ex.one}. The
structural Lie algebra of vector fields is \m{\maV = r \Gamma(T\oM)=}
the space of vector fields on \m{\oM} that vanish on the
boundary. Using the same notation as in the Example \eqref{ex.one},
near a point of the boundary \m{\pa \oM = \{ r = 0\}}, a local basis
is given by
\begin{equation} 
  r \pa_r,\, r \pa_{y_2},\, \ldots ,\, r \pa_{y_n} \,,
\end{equation}
so \m{\maV} is a finitely generated, projective \m{\CI(\oM)}--module.
Since \m{\maV} is also closed under the Lie bracket and \m{\Gamma_c(M;
  TM) \subset \maV \subset \maV_b}, we have that \m{(\oM, \maV)}
defines indeed a Lie manifold.

The orbits \m{Z_\alpha \subset \pa \oM} are reduced to points, so we
can take \m{I := \pa \oM}, and \m{G_\alpha = T_\alpha \pa \oM \rtimes
  \RR} is the semi-direct product with \m{\RR} acting by dilations on
the vector space \m{T_\alpha \pa \oM}, $\alpha \in I$.  The
pseudodifferential calculus \m{\Psi_{\maV}\sp{*}(M)}for this example
was defined also by \cite{Lauter}, Lauter-Moroianu
\cite{LauterMoroianu1}, Mazzeo \cite{Mazzeo91}, and Schulze
\cite{SchulzeBook91}. The metric is {\em asymptotically hyperbolic}.
See also \cite{AlbinMelrose1, Gerard, Moroianu2010}.
\end{example}

The following example covers, in particular, \m{\RR\sp{n}} with the
usual Euclidean metric and with the radial compactification.

\begin{example}\label{ex.three}
As in the previous example, \m{\oM} is a compact manifold with smooth
boundary \m{\pa \oM=\{r =0\}}. We shall take now \m{\maV = r \maV_b=}
the space of vector fields on \m{\oM} that vanish on the boundary
\m{\pa \oM} and whose normal covariant derivative to the boundary also
vanishes. Using the same notation as in the previous two examples, at
the boundary \m{\pa \oM}, a local basis is given by
\begin{equation}
  r^2 \pa_r,\, r \pa_{y_2},\, \ldots ,\, r \pa_{y_n}\,.
\end{equation}

Again the orbits \m{Z_\alpha} are reduced to points, so \m{\alpha \in
  I := \pa \oM}, but this time \m{G_\alpha = T_\alpha \oM = T_\alpha
  \pa \oM \times \RR} is commutative. See also \cite{rodinoSG2007,
  MelroseICM, parenti, SchroheFrechet}. If \m{\pa \oM = S^{n-1}}, the
resulting geometry is that of an asymptotically Euclidean manifold. In
particular, \m{\RR\sp{n}} with the radial compactification fits into
the framework of this example.
\end{example}

\begin{example}\label{ex.four}
As in the previous two examples, \m{\oM} is a compact manifold with
smooth boundary \m{\pa \oM=\{r =0\}}. To construct our Lie algebra of
vector fields \m{\maV = \maV_e}, we assume that we are given a smooth
fibration \m{\pi : \pa M \to B}, and we let \m{\maV_e} to be the space
of vector fields on \m{\oM} that {\em are tangent} to the fibers of
\m{\pi : \pa \oM \to B}. By choosing a product coordinate system on a
small open subset of the boundary, a local basis is then given by
\begin{equation}
    r \pa_r,\, r \pa_{y_2}, \ldots , \, r  \pa_{y_k},\, \pa_{y_{k +
      1}},\, \ldots , \, \pa_{y_{n}} \,.
\end{equation}
Here \m{k} is such that the fibers of \m{\pi : \pa \oM \to B} have
dimension \m{n-k}. Thus, when \m{k=1} (so the fibration is over a
point, that is, \m{\pi : \pa \oM \to pt}), we recover our first
example, Example \ref{ex.one.bis}. On the other hand, when \m{k=n} (so
the fibration is \m{\pi : \pa \oM \to \pa \oM}), we recover our second
example, Example \ref{ex.two}. For \m{n=3} and \m{k=2}, we recover the
edge differential operators of Example \ref{ex.edge} (3) (see Equation
\eqref{eq.def.edge}). We note that \m{\maV := \maV_e \subset \maV_b =:
  \maW} yields a typical example of an anisotropic stucture.

In general, in this example, the set of orbits is \m{I = \{\alpha\}=
  B}, \m{Z_\alpha = \pi^{-1}(\alpha)}, and \m{G_{\alpha} = T_{\alpha}B
  \rtimes \RR} is a solvable Lie group with \m{\RR} acting by
dilations. The geometry is related to that of locally symmetric
spaces. Differential operators of this kind appear in the study of
behavior at the edge of boundary value problems. This example
generalizes the second example (Example \ref{ex.two}) and the same
references are valid for this example as well. See however also
\cite{Grushin71} for possibly the first paper on this type of
examples.
\end{example}

We conclude with some less standard examples.

\begin{example}\label{ex.five}
Let us assume that we are in the same framework as in the previous
example, Example, \ref{ex.four}, but we replace the fibration of
\m{\pa\oM} with a foliation.  Then the resulting Lie manifold may fail
to satisfy Theorem \ref{thm.fc.Lie}.  See however \cite{Rochon}.  It
is interesting to notice that in this case, the resulting class of
Riemann manifolds leads naturally to the study of foliation algebras.
\end{example}


Our last example in this subsection is on a manifold with corners.

\begin{example}\label{ex.six} Let \m{A \to \oM} be a Lie algebroid
(we do not assume \m{\Gamma(\oM; A) \subset \maV_b}) and let \m{\phi :
    \oM \to [0, \infty)} be a smooth function such that \m{ \{ \phi =
      0 \} = \pa \oM}. We define \m{\maV := \phi \Gamma(\oM; A)}. Then
    \m{(\oM, \maV)} defines a Lie manifold.
\end{example}

A related example deals with the $N$-body problem in Quantum Mechanics
\cite{Derezinski-Gerard} and can be used to give a new proof of the
classical HWZ-theorem on the essentials spectrum of these
operators. This is too long and technical to include here, however.

\subsection{Index theory}


Let now \m{(\oM, A)} be a Lie manifold and let \m{f} be the product of
the defining functions of all its faces. We consider then the exact
sequence
\begin{equation}\label{eq.gen.es}
	0 \to f \Psi^{-1}(\maG) \to \Psi^0(\maG) \to \Symb \to
          0 \,,
\end{equation}
which gives rise as before to the map $ \pa : K_1(\Symb) \to K_0(I).$
The {\bf Fredholm index problem} is in this case to compute
\begin{equation*}
	 Tr_* \circ \pa :  K_1(\Symb) \to \ZZ \,.
\end{equation*}
Since \m{ \phi_* \circ \pa = \psi_*}, where \m{\psi = \pa \phi \in
  \HP^1(\Symb)}, by Connes' results, the {\em Fredholm index problem}
is equivalent to computing the class of \m{\psi} in periodic cyclic
homology. This is a difficult problem that is still largely unsolved.
Undoubtedly, excision in cyclic theory will play an important role
\cite{CuntzQuillen0}. See also \cite{ConnesNCG, ConnesBook,
  ConnesMoscovici, NistorFol, NistorDocumenta, Perrot,
  PerrotRodsphon, Pflaum1, Rodsphon, Tsygan}. Instead of this
general problem, we shall look now at a particular, but relevant case
\cite{Carvalho, CarvalhoNistor}.

\begin{definition}
 We say that a be a Lie manifold \m{(\oM, \maW)} is {\em
   asymptotically commutative} if all vectors in \m{\maW} vanish on
 \m{\pa \oM} and all isotropy Lie algebras \m{\ker(q_x)} are
 commutative.
\end{definition}
 
Let \m{x_1, x_2, \ldots, x_k} be the defining functions of all the
hyperfaces of \m{\oM} and \m{f = x_1^{a_1} x_2^{a_2} \ldots x_k^{a_k}}
for some positive integers \m{a_j}. Then, for any Lie manifold, the
product \m{(\oM, \maV)} \m{\maW := f \maV} defines an asymptotically
commutative Lie manifold \m{(\oM, \maW)}.


If \m{(\oM, \maW)} is asymptotically commutative, then the algebra
$\Symb$ is commutative. Its completion will be of the form
\m{C(\Omega)}, as in the work of Cordes \cite{CordesBook, CordesOwen}.
Since the algebra \m{\Symb} is commutative in this case, it is
possible then to compute the index of Fredholm operators using
classical invariants \cite{CarvalhoNistor}. As an application, one
obtains also the index of Dirac operators coupled with potentials of
the form \m{f^{-1} V_0}, where \m{V_0} is invertible at infinity on
any Lie manifold (not just asymptotically commutative)
\cite{CarvalhoNistor}.

\subsection{Essential spectrum}
We now present some applications to essential spectra.  We use the
notation introduced in Subsection \ref{ssec.comparison}.  The
applications to essential spectra of operators are based on the fact
that for a self-adjoint operator \m{D} we have that
\begin{equation}
 \lambda \in \sigma_{ess}(D) \ 
 \Leftrightarrow \ D - \lambda \mbox{ is not Fredholm.}
\end{equation}
We shall consider the case of self-adjoint operators $D$ affiliated to
\m{\mfk{A}(M, \maV)} (that is, satisfying \m{(D + \imath)\sp{-1} \in
  \mfk{A}(M, \maV)} \cite{DamakGeorgescu, GeorgescuIftimovici}).  Then
we can use Theorem \ref{thm.fc.Lie} to study when \m{D - \lambda} is
(or is not) Fredholm.

We shall use these ideas for the open manifolds modeled by the Lie
algebra of vector fields \m{(\oM, \maV_b)} of Example
\eqref{ex.one.bis} and the associated positive Laplacian \m{\Delta_M}
\cite{LauterNistor}.

\begin{theorem} 
Let \m{\oM} be a manifold with corners and let \m{(\oM, \maV_b)} be
the Lie manifold of Example \eqref{ex.one.bis}.  We endow $M$, the
interior of $\oM$, with the induced metric.  Let \m{\Delta_M} be the
associated positive Laplacian on \m{M}. Assuming that $M \neq \oM$, we
have
\begin{equation*}
	\sigma(\Delta_{M}) \, = \, [0,\infty)
\end{equation*}
\end{theorem}

A complete characterization of the spectrum (multiplicity of the
spectral measure, discreteness of the point spectrum, absence of
continuous singular spectrum) is wide open, in spite of its
importance.

Similarly, let \m{\Dir} be the Dirac operator associated to a
\m{Cliff(A)}-bundle over \m{\oM}. Then \cite{NistorPolyhedral}

\begin{theorem} 
The Dirac operator \m{\Dir} on \m{M = \oM \smallsetminus \pa \oM} is
invertible if, and only if, for any open face $F$ (including the
interior face $M$), the associated Dirac operator \m{\Dir_F}, has no
harmonic spinors (that is, it has zero kernel).
\end{theorem}

The proof uses Theorem \ref{thm.fc.Lie} and the fact that the
resulting operators \m{P_\alpha} are also Dirac operators.

Many similar results were obtained in Quantum Mechanics by Georgescu
and his collaborators \cite{DamakGeorgescu, GeorgescuIftimovici,
  GeorgescuNistor1}. In fact, certain problems related to the
\m{N}--body problem can be formulated in terms of a suitable
compactifications of \m{X := \RR\sp{3n}} to a manifold with corners
\m{\oM} on which \m{X} still acts and such that the Lie algebra of
vector fields \m{\maV} is obtained from the action of \m{X}
\cite{GeorgescuNistor2}. See also \cite{Derezinski-Gerard}.

\subsection{Hadamard well posedness on polyhedral domains}
This type of application \cite{BMNZ} is of a different nature and does
not use pseudodifferential operators or other operator algebras. It
uses only Lie manifolds and their geometry. Let then \m{\Omega \subset
  \RR\sp{n}} be an {\em open, bounded} subset of with boundary \m{\pa
  \Omega}. We shall consider the ``simplest'' boundary value problem
on \m{\Omega}, the {\em Poisson problem} with {\em Dirichlet} boundary
conditions:
\begin{equation}
    \begin{cases}  
    \ \  - \Delta u \, = \, f \, & \\
     \ \ \ u \vert_{\pa \Omega} \, = \, 0 \,. &
\end{cases} 
\end{equation}
We refer to \cite{BMNZ} for further references and details not
included here. Recall then the following classical result, which we
shall refer to as the {\em basic well-posedness theorem (for
  \m{\Delta} on smooth domains)}

\begin{theorem} 
Let us assume that \m{\pa\Omega} is smooth. Then the Laplacian
\m{\Delta} defines an isomorphism
\begin{equation*} 
     \Delta : H^{s+1}(\Omega) \cap \{ u \vert _{\pa \Omega} = 0 \} \to
     H^{s-1}(\Omega), \quad s \ge 0 \,.
\end{equation*}
\end{theorem}

A useful consequence (easy to contradict for non-smooth domains) is:

\begin{corollary}\ 
If \m{f} and \m{\pa \Omega} are smooth, then the solution \m{u} of the
Poisson problem with Dirichlet boundary conditions is also smooth.
\end{corollary}

It has been known for a very long time that the basic well posedness
theorem {\em does not extend} to the case when \m{\pa \Omega} is {\em
  non-smooth}.  This can be immediately seen from the following
example.

\begin{example} Let us assume that \m{\Omega} is the unit square, that is
\m{\Omega = (0,1)^2}. If \m{u} is smooth, then \m{\pa_x^2u(0,0) = 0 =
  \pa_y^2u(0,0)}, and hence \m{f(0,0) = \Delta u(0,0) = 0}. By
choosing \m{f(0,0) \neq 0}, we will thus obtain a solution \m{u} that
is not smooth.
\end{example}

In view of the many practical applications of the basic well-posedness
theorem, we want to extend it in some form to non-smooth domains.
Assume now \m{\Omega \subset \RR\sp{n}} is a {\em polyhedral domain.}
Exactly what a polyhedral domain means in three dimensions is subject
to debate. In this presentation, we shall use the definition in
\cite{BMNZ} in terms of stratified spaces (we refer to that paper--a
version of which paper was first circulated in 2004 as an IMA
preprint--for the exact definition). The key technical point in that
paper is to replace the classical Sobolev spaces \m{H^m(\Omega)},
introduced in Equation \eqref{eq.def.Sobolev} with weighted versions
as in Kondratiev\rp s paper \cite{Kondratiev67}.  Let us then denote
by \m{\rho} the distance function to the singular part of the boundary
and define
\begin{equation*}
  \Kond{m}{a}(\Omega) \, := \, \{u ,\, \rho^{|\alpha|-a}
    \pa^\alpha u \in L^2(\Omega),\ |\alpha| \le m\} \,.
\end{equation*}
(Notice the appearance of the factor \m{\rho}!)  Thus, in two
dimensions, \m{\rho(x)} is the distance from \m{x \in \Omega} to the
vertices of \m{\Omega}, whereas in three dimensions, \m{\rho(x)} is
the distance from \m{x \in \Omega} to the set of edges of \m{\Omega}.


\begin{theorem}\label{thm.polyhedral}
Let \m{\Omega \subset \RR^n} be a bounded polyhedral domain and \m{m
  \in \ZZ_+}.  Then there exists \m{\eta > 0} such that
\begin{equation*}
    \Delta \, : \, \Kond{m + 1}{a+1}(\Omega) \cap \{u \vert_{\pa \Omega} = 0
    \} \ \to \ \Kond{m - 1}{a - 1}(\Omega) \,,
\end{equation*}
is an isomorphism for all \m{|a|< \eta}.
\end{theorem}

In two dimensions, this result is due to Kondratiev
\cite{Kondratiev67}.

The proof of Theorem \ref{thm.polyhedral} is based on a study of the
properties of a {\em Lie manifold with boundary} \m{\Sigma(\Omega)}
canonically associated to \m{\Omega} by a {\em blow-up} procedure.
The weighted Sobolev spaces \m{\maK_a^m(\Omega)} can be shown to
coincide with the usual Sobolev spaces associated to
\m{\Sigma(\Omega)}. See \cite{sobolev} for the definition of Lie
manifolds with boundary. General blow-up procedures for Lie manifolds
were studied in \cite{schroedinger}. It can be shown that the class of
Lie manifolds satisfying Theorem \ref{thm.fc.Lie} is closed under
blow-ups with respect to tame Lie submanifolds. Since most practical
applications deal with Lie manifolds that are obtained by such a
blow-up procedure from a smooth manifold, that establishes Theorem
\ref{thm.fc.Lie} in most cases of interest.

The blow-up procedure is an inductive procedure that consists in
successively replacing cones of the form \m{ CL := [0, \epsilon) \times
    L/ (\{0\} \times L)} with their associated cylinders \m{[0,
      \epsilon) \times L.} 

No well posedness result similar to Theorem \ref{thm.polyhedral} holds
for the {\em Neumann problem} (normal derivative at the boundary is
zero):
\begin{equation}
    \begin{cases}  
    \ - \Delta u \, = \, f \ & \\ \ \ \pa_\nu u \, = \, 0 \,, &
\end{cases} 
\end{equation}
where \m{\nu} is a continuous unit normal vector field at the
boundary. In fact, in three dimensions, the above problem is never
Fredholm.

Here is however a variant of Theorem \ref{thm.polyhedral} that has
been proved proved useful in practice. Let us consider a polygonal
domain \m{\Omega} and, for each vertex $P$ of $\Omega$, let us
consider a function \m{\chi_P \in \CI(\RR\sp{2})} that is equal to 1
around the vertex \m{P}, depends only on the distance to \m{P}, and
has small support. Let \m{W_s} be the linear span of the functions
\m{\chi_P}, where \m{P} ranges through the set of vertices of
\m{\Omega}.  Let \m{\{1\}\sp{\perp}} be the space of functions with
integral zero. Then we have the following result \cite{HMN}.

\begin{theorem}\label{thm.neumann}
Let \m{\Omega \subset \RR^2} be a connected, bounded polygonal domain
and \m{m \in \ZZ_+}. Then there exists \m{\eta > 0} such that
\begin{equation*}
    \Delta \, : \, \Big (\, \Kond{m + 1}{a+1}(\Omega) \cap \{\pa_\nu u
     = 0 \}\, + \, W_s \, \Big) \cap \{1\}\sp{\perp}
    \, \to \, \Kond{m - 1}{a - 1}(\Omega) \cap \{1\}\sp{\perp} \,,
\end{equation*}
is an isomorphism for all \m{0 < a < \eta}.
\end{theorem}

The proof of this theorem is based on an index theorem on polygonal
domains, more precisely, a relative index theorem as follows.

\begin{proof} (Sketch) Let us
denote by \m{\Delta_a} the operator for a fixed value of the weight
\m{a}. Then one knows by \cite{Kondratiev67} (or an analysis similar
to the one needed for the APS index formula), that \m{\Delta_a} is
Fredholm if, and only if, \m{a \neq k \pi/\alpha}, where \m{k \in \ZZ}
and \m{\alpha} ranges through the values of the angles of our domain
\m{\Omega}. (For the Dirichlet problem one has a similar condition for
\m{a}, except that \m{k \neq 0}.) One sees that \m{\Delta_0} is not
Fredholm, but one can compute the relative index \m{\ind(\Delta_a) -
  \ind (\Delta_{-a}) = -2n}, for \m{a>0} small, where \m{n} is the
number of vertices of \m{\Omega} (anyone familiar with the APS theory
will have no problem proving this crucial fact). By definition,
\m{\Delta_{a}\sp{*} = \Delta_{-a}}, and hence this gives
\m{\ind(\Delta_{a}) = -n}. A standard energy estimate shows that
\m{\ker(\Delta_{a}) = 1} for \m{a > 0}, with the kernel given by the
constants. This is enough to complete the proof.
\end{proof}

Theorems \ref{thm.polyhedral} and \ref{thm.neumann} have found
applications to optimal rates of convergence for the Finite Element
Method in two and three dimensions \cite{BNZ3D2}, where optimal rates
of convergence in three dimensions were obtained for the first time.

\def\cprime{$'$}

\end{document}